\documentclass[twoside]{amsart}
\usepackage{stmaryrd}

\usepackage{amsmath,amsfonts,amsthm,mathrsfs}
\usepackage[unicode,linktocpage]{hyperref}

\usepackage[alphabetic,initials]{amsrefs}

\usepackage{amssymb}
\usepackage{enumitem}

\usepackage{tikz}
\usetikzlibrary{matrix,arrows}
\newlength{\myarrowsize} 

\pgfarrowsdeclare{myto}{myto}{
	\pgfsetdash{}{0pt} 
	\pgfsetbeveljoin 
	\pgfsetroundcap 
	\setlength{\myarrowsize}{0.6pt}
	\addtolength{\myarrowsize}{.5\pgflinewidth}
	\pgfarrowsleftextend{-4\myarrowsize-.5\pgflinewidth} 
	\pgfarrowsrightextend{.7\pgflinewidth}
}{
	\setlength{\myarrowsize}{0.6pt} 
  	\addtolength{\myarrowsize}{.5\pgflinewidth}  
	\pgfsetroundjoin
	\pgfsetlinewidth{0.0001pt}
	\pgfpathmoveto{\pgfpoint{0.43\myarrowsize}{0}}
	\pgfpatharc{0}{70}{0.14\myarrowsize}
	\pgfpatharc{-110}{-169.5}{4\myarrowsize}
	\pgfpatharc{10.5}{189}{0.25\myarrowsize and 0.12\myarrowsize}
	\pgfpatharc{-170}{-119.5}{4.48\myarrowsize}
	\pgfpathmoveto{\pgfpoint{0.43\myarrowsize}{0}}
	\pgfpatharc{0}{-70}{0.14\myarrowsize}
	\pgfpatharc{110}{169.5}{4\myarrowsize}
	\pgfpatharc{-10.5}{-189}{0.25\myarrowsize and 0.12\myarrowsize}
	\pgfpatharc{170}{119.5}{4.48\myarrowsize}
	\pgfpathclose
	\pgfsetstrokeopacity{0.25}
	\pgfusepathqfillstroke
}

\pgfarrowsdeclare{myonto}{myonto}{
	\pgfsetdash{}{0pt} 
	\pgfsetbeveljoin 
	\pgfsetroundcap 
	\setlength{\myarrowsize}{0.6pt}
	\addtolength{\myarrowsize}{.5\pgflinewidth}
	\pgfarrowsleftextend{-4\myarrowsize-.5\pgflinewidth} 
	\pgfarrowsrightextend{.7\pgflinewidth}
}{
	\setlength{\myarrowsize}{0.6pt} 
  	\addtolength{\myarrowsize}{.5\pgflinewidth}  
	\pgfsetroundjoin
	\pgfsetlinewidth{0.0001pt}
	\pgfpathmoveto{\pgfpoint{0.43\myarrowsize}{0}}
	\pgfpatharc{0}{70}{0.14\myarrowsize}
	\pgfpatharc{-110}{-169.5}{4\myarrowsize}
	\pgfpatharc{10.5}{189}{0.25\myarrowsize and 0.12\myarrowsize}
	\pgfpatharc{-170}{-119.5}{4.48\myarrowsize}
	\pgfpathlineto{\pgfpoint{0.43\myarrowsize-0.3em}{0}}
	\pgfpatharc{0}{70}{0.14\myarrowsize}
	\pgfpatharc{-110}{-169.5}{4\myarrowsize}
	\pgfpatharc{10.5}{189}{0.25\myarrowsize and 0.12\myarrowsize}
	\pgfpatharc{-170}{-119.5}{4.48\myarrowsize}
	\pgfpathmoveto{\pgfpoint{0.43\myarrowsize}{0}}
	\pgfpatharc{0}{-70}{0.14\myarrowsize}
	\pgfpatharc{110}{169.5}{4\myarrowsize}
	\pgfpatharc{-10.5}{-189}{0.25\myarrowsize and 0.12\myarrowsize}
	\pgfpatharc{170}{119.5}{4.48\myarrowsize}
	\pgfpathlineto{\pgfpoint{0.43\myarrowsize-0.3em}{0}}
	\pgfpatharc{0}{-70}{0.14\myarrowsize}
	\pgfpatharc{110}{169.5}{4\myarrowsize}
	\pgfpatharc{-10.5}{-189}{0.25\myarrowsize and 0.12\myarrowsize}
	\pgfpatharc{170}{119.5}{4.48\myarrowsize}
	\pgfpathclose
	\pgfsetstrokeopacity{0.25}
	\pgfusepathqfillstroke
}

\pgfarrowsdeclare{myhook}{myhook}{
	\setlength{\myarrowsize}{0.6pt}
	\addtolength{\myarrowsize}{.5\pgflinewidth}
	\pgfarrowsleftextend{-4\myarrowsize-.5\pgflinewidth} 
	\pgfarrowsrightextend{.7\pgflinewidth}
}{
	\setlength{\myarrowsize}{0.6pt} 
  	\addtolength{\myarrowsize}{.5\pgflinewidth}  
 	\pgfsetdash{}{+0pt}
	\pgfsetroundcap
	\pgfpathmoveto{\pgfqpoint{0pt}{-4.667\pgflinewidth}}
	\pgfpathcurveto
    {\pgfqpoint{4\pgflinewidth}{-4.667\pgflinewidth}}
    {\pgfqpoint{4\pgflinewidth}{0pt}}
    {\pgfpointorigin}
	\pgfusepathqstroke
}

\newenvironment{diagram}[2]{%
\[%
\begin{tikzpicture}[>=myto,baseline=(current bounding box.center),%
	to/.style={->,font=\scriptsize,cap=round},%
	into/.style={myhook->,font=\scriptsize,cap=round},%
	onto/.style={-myonto,font=\scriptsize,cap=round},%
	math/.style={matrix of math nodes, row sep=#2, column sep=#1,%
		text height=1.5ex, text depth=0.25ex}]%
}{%
\end{tikzpicture}%
\]%
\ignorespacesafterend%
}

\usepackage{chngcntr}

\makeatletter
\let\@@seccntformat\@seccntformat
\renewcommand*{\@seccntformat}[1]{%
  \expandafter\ifx\csname @seccntformat@#1\endcsname\relax
    \expandafter\@@seccntformat
  \else
    \expandafter
      \csname @seccntformat@#1\expandafter\endcsname
  \fi
    {#1}%
}
\newcommand*{\@seccntformat@subsection}[1]{%
  \textbf{\csname the#1\endcsname.}
}
\makeatother

\counterwithin{equation}{subsection}
\counterwithout{subsection}{section}
\counterwithin{figure}{subsection}

\newcommand{\subsecref}[1]{\S\ref{#1}}

\newtheorem{theorem}[equation]{Theorem}
\newtheorem*{theorem*}{Theorem}
\newtheorem{lemma}[equation]{Lemma}
\newtheorem*{lemma*}{Lemma}
\newtheorem{corollary}[equation]{Corollary}
\newtheorem{proposition}[equation]{Proposition}
\newtheorem*{proposition*}{Proposition}

\theoremstyle{definition}
\newtheorem{definition}[equation]{Definition}
\newtheorem{problem}[equation]{Problem}
\newtheorem*{definition*}{Definition}
\theoremstyle{remark}

\newtheorem{example}[equation]{Example}
\newtheorem*{example*}{Example}
\newtheorem*{note}{Note}

\theoremstyle{plain}

\newcommand{\MHM}{\operatorname{MHM}}

\newcommand{\pt}{\mathit{pt}}
\newcommand{\Dmod}{\mathscr{D}}
\newcommand{\Mmod}{\mathcal{M}}
\newcommand{\Nmod}{\mathcal{N}}

\newcommand{\derR}{\mathbb{R}}
\newcommand{\derL}{\mathbb{L}}

\newcommand{\decal}[1]{\lbrack #1 \rbrack}

\newcommand{\norm}[1]{\lVert#1\rVert}

\newcommand{\abs}[1]{\lvert #1 \rvert}

\newcommand{\eps}{\varepsilon}

\newcommand{\tensor}{\otimes}

\newcommand{\shHom}{\mathscr{H}\hspace{-2.7pt}\mathit{om}}
\newcommand{\shExt}{\mathscr{E}\hspace{-1.5pt}\mathit{xt}}
\newcommand{\Hom}{\operatorname{Hom}}

\newcommand{\NN}{\mathbb{N}}
\newcommand{\ZZ}{\mathbb{Z}}
\newcommand{\QQ}{\mathbb{Q}}
\newcommand{\RR}{\mathbb{R}}
\newcommand{\CC}{\mathbb{C}}

\newcommand{\PPn}[1]{\mathbb{P}^{#1}}

\newcommand{\menge}[2]{\bigl\{ \thinspace #1 \thinspace\thinspace \big\vert%
\thinspace\thinspace #2 \thinspace \bigr\}}
\newcommand{\Menge}[2]{\Bigl\{ \thinspace #1 \thinspace\thinspace \Big\vert%
\thinspace\thinspace #2 \thinspace \Bigr\}}

\DeclareMathOperator{\im}{im}

\DeclareMathOperator{\Spec}{Spec}

\DeclareMathOperator{\Tor}{Tor}
\DeclareMathOperator{\id}{id}

\DeclareMathOperator{\Supp}{Supp}
\DeclareMathOperator{\codim}{codim}

\DeclareMathOperator{\rat}{rat}

\DeclareMathOperator{\Sym}{Sym}
\DeclareMathOperator{\Gr}{Gr}
\DeclareMathOperator{\DR}{DR}

\newcommand{\define}[1]{\emph{#1}}

\newcommand{\shf}[1]{\mathscr{#1}}
\newcommand{\OX}{\shf{O}_X}
\newcommand{\OmX}[1]{\Omega_X^{#1}}
\newcommand{\OW}{\shO_W}

\newcommand{\restr}[1]{\big\vert_{#1}}

\newcommand{\argbl}{-}

\newcommand{\dst}{\Delta^{\ast}}

\newcommand{\into}{\hookrightarrow}

\newcommand{\fu}{f^{\ast}}
\newcommand{\fl}{f_{\ast}}

\newcommand{\iu}{i^{\ast}}

\newcommand{\pu}{p^{\ast}}
\newcommand{\pl}{p_{\ast}}

\newcommand{\shF}{\shf{F}}

\newcommand{\shO}{\shf{O}}

\newcommand{\al}{a_{\ast}}
\newcommand{\au}{a^{\ast}}
\newcommand{\OA}{\shf{O}_A}
\newcommand{\OAh}{\shf{O}_{\Ah}}
\newcommand{\OAsh}{\shf{O}_{\Ash}}
\newcommand{\GV}{\operatorname{GV}}
\newcommand{\OV}{\shO_V}
\newcommand{\ql}{q_{\ast}}
\renewcommand{\Gr}{\operatorname{gr}}
\DeclareMathOperator{\Perv}{Perv}
\newcommand{\OY}{\shO_Y}

\newcommand{\Ash}{A^{\natural}}

\newcommand{\Dbcoh}{\mathrm{D}_{\mathrm{coh}}^{\mathrm{b}}}
\newcommand{\Db}{\mathrm{D}^{\mathrm{b}}}
\newcommand{\Dtcoh}[1]{\mathrm{D}_{\mathrm{coh}}^{#1}}
\renewcommand{\derR}{\mathbf{R}}
\renewcommand{\derL}{\mathbf{L}}

\newcommand{\FM}{\derR \Phi_P}
\newcommand{\FMi}{\derR \Psi_P}
\renewcommand{\argbl}{-}

\renewcommand{\Hom}{\operatorname{Hom}}

\renewcommand{\shHom}{\mathcal{H}\mathit{om}}
\newcommand{\Ah}{\widehat{A}}
\DeclareMathOperator{\Alb}{Alb}
\DeclareMathOperator{\Pic}{Pic}

\newcommand{\OmA}[1]{\Omega_A^{#1}}
\newcommand{\DA}{\mathbf{D}_A}

\newcommand{\QHX}{\QQ_X^H}
\newcommand{\omX}{\omega_X}
\newcommand{\shC}{\mathscr{C}}

\newcommand{\F}{\mathcal{F}}
\newcommand{\LL}{\mathcal{L}}

\newcommand{\Coh}{\operatorname{Coh}}

\newcommand{\pDtcoh}[2]{ {^{#1}} \Dtcoh{#2}}
\newcommand{\pCoh}[2]{ {^{#1}} \mathrm{Coh}{#2}}
\newcommand{\QHY}{\QQ_Y^H}
\newcommand{\QHF}{\QQ_F^H}
\newcommand{\shA}{\mathcal{A}}
\newcommand{\AX}{\shA_X}
\newcommand{\AY}{\shA_Y}
\newcommand{\AZ}{\shA_Z}

\newcommand{\Vd}{V^{\ast}}

\newcommand{\DX}{\mathbf{D}_X}

\newcommand{\QHC}{\QQ_C^H}

\newcommand{\QHA}{\QQ_A^H}

\newcommand{\DAV}{\mathbf{D}_{A \times V}}
\newcommand{\DAhV}{\mathbf{D}_{\Ah \times V}}

\newcommand{\mm}{\mathfrak{m}}

\newcommand{\Rh}{\hat{R}}

\renewcommand{\dbar}{\bar{\partial}}
\newcommand{\dbarst}{\bar{\partial}^{\ast}}
\newcommand{\del}{\partial}
\newcommand{\delst}{\partial^{\ast}}
\newcommand{\Har}{\mathcal{H}}
\renewcommand{\dst}{d^{\ast}}

\newcommand{\dbartau}{\dbar_{\tau}}
\newcommand{\dtau}{d_{\tau}}
\newcommand{\deltau}{\del_{\tau}}
\newcommand{\dbarsttau}{\dbarst_{\tau}}
\newcommand{\delsttau}{\delst_{\tau}}
\newcommand{\dsttau}{\dst_{\tau}}
\newcommand{\Laptau}{\Delta_{\tau}}
\newcommand{\Hartau}{\Har_{\tau}}
\newcommand{\Htau}{H_{\tau}}
\newcommand{\Gtau}{G_{\tau}}

\DeclareMathOperator{\HM}{HM}
\DeclareMathOperator{\IC}{IC}

\newcommand{\cCoh}{ {^c} \mathrm{Coh}}
\newcommand{\cDtcoh}[1]{ {^c} \Dtcoh{#1}}
\newcommand{\mCoh}{ {^m} \mathrm{Coh}}
\newcommand{\mDtcoh}[1]{ {^m} \Dtcoh{#1}}

\newcommand{\shT}{\mathscr{T}}

\begin{document}

\title[Generic vanishing theory via mixed Hodge modules]{Generic vanishing theory \\ via mixed Hodge modules}
\author[M.~Popa]{Mihnea Popa}
\address{%
	Department of Mathematics, Statistics \& Computer Science \\
	University of Illinois at Chicago \\
	851 South Morgan Street \\
	Chicago, IL 60607
}
\email{mpopa@math.uic.edu}

\author[Ch.~Schnell]{Christian Schnell}
\address{%
	Institute for the Physics and Mathematics of the Universe \\
	The University of Tokyo \\
	5-1-5 Kashiwanoha, Kashiwa-shi \\
	Chiba 277-8583, Japan
}
\email{christian.schnell@ipmu.jp}


\begin{abstract}
We extend the dimension and strong linearity results of generic vanishing theory to 
bundles of holomorphic forms and rank one local systems, and more generally to
certain coherent sheaves of Hodge-theoretic origin associated to irregular varieties.  
Our main tools are Saito's mixed Hodge modules, the Fourier-Mukai transform for 
$\Dmod$-modules on abelian varieties
introduced by Laumon and Rothstein, and Simpson's harmonic theory for flat bundles.
In the process, we discover two natural categories of perverse coherent sheaves.
\end{abstract}
\maketitle


\makeatletter
\newcommand\@dotsep{4.5}
\def\@tocline#1#2#3#4#5#6#7{\relax
  \ifnum #1>\c@tocdepth 
  \else
    \par \addpenalty\@secpenalty\addvspace{#2}%
    \begingroup \hyphenpenalty\@M
    \@ifempty{#4}{%
      \@tempdima\csname r@tocindent\number#1\endcsname\relax
    }{%
      \@tempdima#4\relax
    }%
    \parindent\z@ \leftskip#3\relax
    \advance\leftskip\@tempdima\relax
    \rightskip\@pnumwidth plus1em \parfillskip-\@pnumwidth
    #5\leavevmode\hskip-\@tempdima #6\relax
    \leaders\hbox{$\m@th
      \mkern \@dotsep mu\hbox{.}\mkern \@dotsep mu$}\hfill
    \hbox to\@pnumwidth{\@tocpagenum{#7}}\par
    \nobreak
    \endgroup
  \fi}
\def\l@section{\@tocline{1}{3pt}{1pc}{}{\bfseries}}
\def\l@subsection{\@tocline{2}{0pt}{30pt}{5pc}{}}
\makeatother

\tableofcontents

\section{Introduction}

\subsection{Generic vanishing theory}

The attempt to understand cohomology vanishing statements on irregular varieties in
the absence of strong positivity has led to what is usually called \emph{generic
vanishing theory}. Perhaps the most famous result is the generic
vanishing theorem of Green and Lazarsfeld \cite{GL1}, which in a weak form states that on a smooth
complex projective variety $X$, the cohomology of a generic line bundle $L \in \Pic^0(X)$
vanishes in degrees less than $\dim a(X)$, where $a \colon X \to \Alb(X)$
denotes the Albanese mapping of $X$.
This theorem and its variants have found a surprising number of applications, ranging from results about
singularities of theta divisors \cite{EL} to recent work on the
birational geometry of irregular varieties, including a proof of Ueno's conjecture \cite{ChH}.

One can consider the set of those line bundles for which the cohomology in a given
degree does not vanish, and thanks to the work of many people, the structure of these sets
is very well understood. This is more precisely the content of generic vanishing theory. 
Denoting, for any coherent sheaf $\F$ on $X$, by
$$V^i (\F) := \{ L \in \Pic^0(X) ~|~ H^i (X, \F \otimes L ) \neq 0\} \subseteq \Pic^0 (X)$$
the $i$-th cohomological support locus of $\F$,  its main statements are the following:

\begin{description}
\item[Dimension (D)]
One has ${\rm codim}~V^i (\omega_X) \ge i - \dim X + \dim a(X)$ for all $i$
\cites{GL1,GL2}. This implies the generic vanishing theorem via Serre duality.

\item[Linearity (L)]
The  irreducible components of each $V^i (\omega_X)$ are torsion translates of
abelian subvarieties of $\Pic^0(X)$ \cites{GL2,Arapura,Simpson2}.

\item[Strong linearity (SL)]
If $p_2 : X\times \Pic^0 (X) \rightarrow \Pic^0 (X)$ is the second projection, and $P$ is a 
Poincar\'e line bundle on $X \times \Pic^0 (X)$, 
then $\derR {p_2}_* P$ is locally around each point quasi-isomorphic to a linear
complex \cite{GL2}. A precise version of this result is known to imply (L), except
for the torsion statement, and based on this also (D).
\end{description}

Analogous results have been considered for the cohomology of local systems, replacing
$\Pic^0 (X)$ by ${\rm Char}(X)$, the algebraic group parametrizing rank one local
systems \cites{Arapura,Simpson,Simpson2}. New approaches and extensions for the
theory on $\Pic^0 (X)$ have been introduced more recently, for example in
\cites{CH,Hacon,PP}. On the other hand, important gaps have remained in our
understanding of some of the most basic objects. For instance, while (L) is also
known for the sheaf of holomorphic $p$-forms $\Omega_X^p$ with $ p < n$, a good
generic Nakano-type vanishing statement as in (D) has eluded previous efforts,
despite several partial results \cites{GL1,PP}. The same applies to the case of local
systems of rank one, where the perhaps the even more interesting property (SL) has been
missing as well.

In this paper, we answer those remaining questions, and at the same time recover the
previous results of generic vanishing theory mentioned above (with the exception of
the statement about torsion points, which is of a different nature) by
enlarging the scope of the study to the class of filtered $\Dmod$-modules
associated to mixed Hodge modules on abelian varieties. In fact, there is a version
of the Fourier-Mukai transform for $\Dmod$-modules, introduced by Laumon \cite{Laumon2} and
Rothstein \cite{Rothstein}; it takes $\Dmod$-modules on an abelian variety to
complexes of coherent sheaves on $\Ash$, the moduli space of line bundles on $A$ with
integrable connection.  Our main results can be summarized briefly as describing
the \emph{Fourier-Mukai transform of the trivial $\Dmod$-module $\OX$} on
an irregular variety $X$.

\subsection{Why mixed Hodge modules?}
\label{subsec:Hacon}

To motivate the introduction of mixed Hodge modules into the problem, let us briefly
recall the very elegant proof of the generic vanishing theorem for $\omega_X$ discovered by Hacon
\cite{Hacon}.  It goes as follows.

Let $A = \Alb(X)$ denote the Albanese variety of an
irregular smooth complex projective variety $X$, and $a \colon X \to A$ its Albanese
mapping (for some choice of base point, which does not matter here). Let $\Ah =
\Pic^0(A)$ denote the dual abelian variety. 
Using a well-known theorem of Koll\'ar on the splitting of the direct image $\derR
\al \omega_X$ in $\Dbcoh(\OA)$ and standard manipulations, it is enough to 
prove that
\[
	\codim V^{\ell} (R^i \al \omega_X) \ge \ell \,\,\,\,\,\,
	{\rm for~ all~} i=0, 1,
\dotsc, k := \dim X - \dim a(X).
\]
In terms of the Fourier-Mukai transform $\derR \Phi_P \colon \Dbcoh(\OA) \to
\Dbcoh(\OAh)$, this is equivalent, via base change arguments, to the statement that 
\[
	\codim \Supp R^{\ell} \Phi_P(R^i \al \omega_X) \geq \ell
\]
in the same range. Now the sheaves $R^i \al \omega_X$ still satisfy a Kodaira-type vanishing theorem, and together with 
the special geometry of abelian varieties this 
implies after some work that the Fourier-Mukai transform of $R^i \al \omega_X$ is the
dual of a coherent sheaf $\shF_i$ on $\Ah$, which is to say that
\[
	\derR \Phi_P(R^i \al \omega_X) \simeq \derR \shHom(\shF_i, \OAh).
\]
The desired inequality for the codimension of the support becomes
\[
	\codim \Supp R^{\ell} \Phi_P(R^i \al \omega_X) = 
		\codim \Supp \shExt^{\ell}(\shF_i, \OAh) \geq \ell,
\]
which is now a consequence of a general theorem about regular local rings. This proves the dimension statement
(D), and hence the generic vanishing theorem for topologically trivial line bundles.

One of the subjects of this paper is to use this framework in order to prove a generic vanishing
theorem for general objects
of Hodge-theoretic origin.  
The role of Koll\'ar's theorem is played by the
decomposition theorem \cite{BBD}, or more precisely by its Hodge-theoretic version
due to Morihiko Saito \cite{Saito-MHM}. This is one main reason why mixed Hodge modules
form a natural setting here. Another is the existence of a very general Kodaira-type
vanishing theorem for mixed Hodge modules, again due to Saito, which becomes
particularly useful on abelian varieties. This vanishing theorem allows us to
generalize the second half of the proof above to any coherent sheaf of
Hodge-theoretic origin on an abelian variety. Finally, in order to extract the
relevant information about the sheaves $\Omega_X^p$ with $p < \dim X$, one needs a result by Laumon and
Saito on the behavior of filtered $\Dmod$-modules under direct images, which only
works well in the case of $\Dmod$-modules that underlie mixed Hodge modules.


\subsection{The main results}

Let us now give a summary of the results we obtain. 
There are essentially two parts: vanishing and dimension results, for which 
Hodge modules are crucially needed, and linearity results, which 
apply to certain Hodge modules, but for which the general theory of 
$\Dmod$-modules and the harmonic theory  of flat line bundles suffice in the proofs.
The theory of mixed Hodge modules is reviewed
in \subsecref{subsec:MHM} below.

The starting point is a general Kodaira-type vanishing theorem for the graded pieces
of the de Rham complex of a mixed Hodge module, proved by Saito. On an abelian
variety $A$, this can be improved to a vanishing theorem for coherent sheaves of the
form $\Gr_k^F \Mmod$, where $(\Mmod, F)$ is any filtered $\Dmod$-module
underlying a mixed Hodge module on $A$ (see Lemma~\ref{lem:kodaira} below). We use
this observation to produce natural classes of perverse coherent sheaves
\cites{AB,Kashiwara} on the dual abelian variety $\Ah$, and on the parameter space
for Higgs line bundles $\Ah \times H^0 (A, \Omega_A^1)$. 

We first show that every mixed Hodge module on $A$ gives rise to a collection of
sheaves satisfying the generic vanishing condition, or equivalently 
perverse coherent sheaves on $\Ah$ with respect to the dual standard $t$-structure (reviewed in 
\S \ref{subsec:t-structure}).

\begin{theorem}\label{hm_abelian}
Let $A$ be a complex abelian variety, and $M$ a mixed Hodge module on $A$ with
underlying filtered  $\Dmod$-module $(\Mmod, F)$. Then for each $k \in \ZZ$, the
coherent sheaf $\Gr_k^F \Mmod$ is a $GV$-sheaf on $A$, meaning that
\[
	\codim V^i (\Gr_k^F \Mmod) \geq i \,\,\,\,{\rm for~all~}i.
\]
Consequently, its Fourier-Mukai transform $\FM(\Gr_k^F \Mmod)$ is a perverse coherent sheaf on $\Ah$.
\end{theorem}

This uses Hacon's general strategy, as in
\subsecref{subsec:Hacon} above, and the correspondence established in
\cites{Popa,PP} between objects satisfying generic vanishing (or $\GV$-objects) and
perverse coherent sheaves in the above sense. 

In order to obtain a generic Nakano-type vanishing statement similar to (D), or
statements for cohomological support loci of rank one local systems, we apply Theorem
\ref{hm_abelian} to the direct image of the trivial Hodge module on an irregular
variety under the Albanese map. Here our main tools are the decomposition 
theorem for Hodge modules \cite{Saito-HM}, extending the well-known result of
\cite{BBD}, and a formula due to Laumon \cite{Laumon} for the behavior of
the associated graded objects under projective direct images (which is true for mixed
Hodge modules).

Our main results in this direction are the following. Let $X$ be a smooth complex projective
variety of dimension $n$, with nonzero irregularity $g = h^1(X, \OX)$. 
Let $a \colon X \to A  = \Alb(X)$ be the Albanese map of $X$.  
Consider the \define{defect of semismallness} of the
Albanese map $a \colon X \to A$, which is defined by the formula
\[
	\delta(a) = \max_{\ell \in \NN} \bigl( 2 \ell - \dim X + \dim A_{\ell} \bigr),
\]
where $A_{\ell} = \menge{y \in A}{\dim f^{-1}(y) \geq \ell}$ for $\ell \in \NN$.
By applying the results quoted above, together with Theorem \ref{hm_abelian},  to the direct
image of the trivial Hodge module $\QHX \decal{n}$, we obtain the following theorem of type
(D) for bundles of holomorphic forms.

\begin{theorem}\label{thm:nakano}
Let $X$ be a smooth complex projective variety of dimension $n$. Then
\[
 	\codim V^q(\Omega_X^p) \ge \abs{p + q - n} - \delta(a).
\]
for every $p,q \in \NN$, and there exists $p$ and $q$ for which this is an equality.
\end{theorem}

The statement above is the appropriate generalization of the original generic vanishing theorem \cite{GL1}, 
which dealt with the case $p=n$. Note that, unlike in \cite{GL1}, the
codimension bound depends on the entire Albanese mapping, not just on the dimension
of the generic fiber. In the language of \cite{PP}, our theorem is equivalent to the
fact that, for each $p \in \ZZ$, the bundle $\Omega_X^p$ is a \emph{$\GV_{p - n -
\delta(a)}$-sheaf} with respect to the Fourier-Mukai transform induced by the Poincar\'e
bundle on $X \times \Pic^0 (X)$. 
Since the condition $\delta (a) = 0$ is equivalent to the Albanese map being semismall,
one obtains in particular:

\begin{corollary}
Suppose that the Albanese map of $X$ is semismall. Then
\[
	\codim V^q (\Omega_X^p)  \geq \abs{p + q - n},
\]
for every $p,q \in \NN$, and so $X$ satisfies the generic Nakano vanishing theorem.
\end{corollary}

Unlike in the case of $\omega_X$, it is not sufficient to assume that the Albanese
map is generically finite over its image; this was already pointed out in \cite{GL1}.
Nevertheless, our method also recovers the stronger statement for $\omega_X$
\cite{GL1} and its higher direct images \cite{Hacon} (see the end of \S10). 
This is due to the special
properties of the first non-zero piece of the Hodge filtration on mixed Hodge
modules, established by Saito. Note also that, since $\chi (\Omega_X^p) = \chi (\Omega_X^p \otimes L)$
for any $L \in \Pic^0 (X)$ by the deformation invariance of the Euler characteristic of
a coherent sheaf, we have as a consequence the following extension of the fact that $\chi (\omega_X) \ge 0$
for varieties of maximal Albanese dimension.

\begin{corollary}
If the Albanese map of $X$ is semismall, then $(-1)^{n-p} \chi (\Omega_X^p) \ge 0$.
\end{corollary}

Our approach also leads to a dimension theorem of type (D) for rank one local
systems.  Let ${\rm Char} (X) = {\rm Hom} ( \pi_1 (X), \CC^{\ast})$ be the algebraic
group of characters of $X$, and for each $i$ consider the cohomological support loci 
\[
\Sigma^k (X) = \menge{\rho \in {\rm Char}(X)}{H^k (X, \CC_{\rho}) \neq 0},
\]
where $\CC_{\rho}$ denotes the local system of rank one associated to a character $\rho$.

\begin{theorem}\label{thm:locsyst}
Let $X$ be a smooth complex projective variety of dimension $n$. Then 
\[
	\codim_{{\rm Char}(X)} \Sigma^k(X) \ge 2 \bigl( \abs{k - n} - \delta(a) \bigr)
\]
for each $k \in \NN$.
\end{theorem}

To deduce this from the arguments leading to Theorem \ref{thm:nakano}, we need to
appeal to the structure results and the relationship with the space of Higgs
bundles, proved by Simpson \cites{Simpson,Simpson2} and Arapura \cite{Arapura}; see \S12.

\begin{note}
While editing this paper, we learned of the very interesting preprint \cite{KW} by
T.~Kr\"amer and R.~Weissauer, who prove vanishing theorems for perverse sheaves on
abelian varieties. They also obtain a generic vanishing theorem for $\Omega_X^p$
and for rank one local systems, involving the same quantity $\delta(a)$ as in 
Theorem~\ref{thm:nakano}, but without precise codimension bounds for the
cohomological support loci. Their methods are very different from ours.
\end{note}

Three additional theorems complete the picture, by describing in detail the
Fourier-Mukai transform of the trivial $\Dmod$-module $\OX$; they include results of
type (D), (L) and (SL) on the space of line bundles with integrable connection on $X$. Here it is
important to consider two different kinds of Fourier-Mukai transforms, corresponding
in Simpson's terminology \cite{Simpson2} to the Dolbeault realization (via Higgs
bundles) and the de Rham realization (via line bundles with integrable connection)
of ${\rm Char}(X)$.

Setting $V= H^0 (A, \Omega_A^1)$, one can naturally extend the
usual Fourier-Mukai functor $\Dbcoh(\OA) \to \Dbcoh(\OAh)$ to a
relative transform (see \S9)
\[
	\FM \colon \Dbcoh \bigl( \shO_{A \times V} \bigr) \to 
		\Dbcoh \bigl( \shO_{\Ah \times V} \bigr).
\]
The first of the three theorems describes how the complex of filtered $\Dmod$-modules
$\al(\OX, F)$ underlying $\al \QHX \decal{n} \in \Db \MHM(A)$ behaves with respect to
this transform. We show in Proposition~\ref{prop:complex} below that the
associated graded complex satisfies
\[
	\Gr_{\bullet}^F  a_* (\OX, F) \simeq 
		\derR a_* \Bigl\lbrack 
			\OX \tensor S^{\bullet-g} \to \OmX{1} \tensor S^{\bullet-g+1} \to \dotsb \to
			\OmX{n} \tensor S^{\bullet-g+n} 
		\Bigr\rbrack,
\]
where $S = \Sym V^{\ast}$, and the complex in brackets is placed in degrees $-n,
\dotsc, 0$, with differential induced by the evaluation morphism $\OX \tensor V \to \OmX{1}$.  
Since this is a complex of finitely generated graded modules over 
${\rm Sym} ~\shT_A$, it naturally corresponds to a complex of coherent sheaves 
on cotangent bundle $T^*A = A\times V$, namely
\[
\shC = \derR (a \times \id)_* \Bigl\lbrack 
		p_1^* \OX \to p_1^* \Omega_X^1 \to \dotsb \to p_1^* \Omega_X^n
	\Bigr\rbrack.
\]

\begin{theorem} \label{thm:mCoh_intro}
Let $a \colon X \to A$ be the Albanese map of a smooth complex projective variety of
dimension $n$, and let $p_1 \colon X \times V \to X$ be the first projection. 
\begin{enumerate}[label=(\roman{*}), ref=(\roman{*})]
\item In the derived category $\Dbcoh(\shO_{A \times V})$, we have a non-canonical
splitting
\[
	\shC \simeq \bigoplus_{i,j} \shC_{i,j} \decal{-i}
\]
where each $\shC_{i,j}$ is a Cohen-Macaulay sheaf of dimension $\dim A$.
\item \label{en:thm2}
The support of each $\FM \shC_{i,j}$ is a finite union of torsion translates of
triple tori in $\Ah \times V$, subject to the inequalities
\[
	\codim \Supp R^{\ell} \Phi_P \shC_{i,j} \geq 2 \ell 
		\qquad \text{for all $\ell \in \ZZ$}.
\]
\item The dual objects $\derR \shHom \bigl( \FM \shC_{i,j}, \shO_{\Ah \times V} \bigr)$
also satisfy \ref{en:thm2}.
\end{enumerate}
\end{theorem}

A slightly more refined version of this can be found in Theorem \ref{thm:mCoh} below, where we
also interpret the result in terms of perverse coherent sheaves: the objects $\FM
\shC_{i,j}$ belong to a self-dual subcategory of the category of perverse coherent 
sheaves, with respect to a certain $t$-structure on $\Dbcoh(\shO_{A \times
V})$ introduced in \subsecref{subsec:t-structure}. Note that this $t$-structure is
different from the dual standard $t$-structure that appears in the generic
vanishing theory of topologically trivial line bundles \cite{Popa}. 

The second and third theorems are best stated in terms of the generalized
Fourier-Mukai transform for $\Dmod$-modules on abelian varieties, introduced by
Laumon \cite{Laumon2} and Rothstein \cite{Rothstein}. Their work gives an equivalence
of categories
\[
	\derR \Phi_{P^\natural} : \Dbcoh (\Dmod_A) \to \Dbcoh(\OAsh),
\]
where $A^\natural$ is the moduli space of line bundles with integrable connection on
$A$, with universal cover $W \simeq H^1(A, \CC) \simeq H^1(X, \CC)$. 
By composing with the Albanese mapping, there is also an induced functor
\[	
	\derR \Phi_{P^\natural} : \Dbcoh (\Dmod_X) \to \Dbcoh (\OAsh).
\]
We review the construction, both algebraically and analytically, in
\subsecref{subsec:Rothstein} below. This Fourier-Mukai transform is the right context
for a strong linearity result (SL) for the $\Dmod$-module $\OX$, extending the result
for topologically trivial line bundles in \cite{GL2}.

\begin{theorem} \label{thm:linear_intro}
Let $X$ be a smooth complex projective variety. Then the generalized Fourier-Mukai
transform $\derR \Phi_{P^\natural} \OX \in \Dbcoh(\OAsh)$ is locally, in the analytic
topology, quasi-isomorphic to a linear complex. More precisely, let $(L, \nabla)$ 
be any line bundle with integrable connection on $X$, and $\LL = \ker \nabla$ the
associated local system. 
Then in an open neighborhood of the origin in $W$, the pullback of $\derR \Phi_{P^\natural}
\OX$ to the universal covering space $W$ is quasi-isomorphic to the complex
\[
	H^0(X, \LL) \tensor \OW \to H^1(X, \LL) \tensor \OW \to 
		\dotsb \to H^{2n}(X, \LL) \tensor \OW,
\]
placed in degrees $-n, \dotsc, n$, with differential given by the formula
\[
	H^k(X, \LL) \tensor \OW \to H^{k+1}(X, \LL) \tensor \OW, \quad
		\alpha \tensor f \mapsto \sum_{j=1}^{2g} (\eps_j \wedge \alpha) \tensor z_j f. 
\]
Here $\eps_1, \dotsc, \eps_{2g}$ is any basis of $H^1(X, \CC)$, and $z_1,
\dotsc, z_{2g}$ are the corresponding holomorphic coordinates on the affine space $W$.
\end{theorem}

As discussed in \S\ref{subsec:SS}, every direct summand of a linear complex (in the
derived category) is again quasi-isomorphic to a linear complex. It follows that
all direct summands of $\derR \Phi_{P^\natural} \OX$ coming from the decomposition
\[
	a_* \QHX \decal{n} \simeq \bigoplus_i M_{i,j} \decal{-i} \in \Db \MHM(A)
\]
have the same linearity property (see Corollary \ref{summands}). Note also that using 
base change for local systems and the description of tangent cones to cohomological support loci as in \cite{Libgober}, 
via arguments as in\cite{GL2}*{\S4} (which we will not repeat here), Theorem \ref{thm:linear_intro} gives another proof of 
the linearity to the cohomological support loci $\Sigma^k (X)$, i.e.~for the statement of type (L). 

Our proof of Theorem~\ref{thm:linear_intro} relies on the harmonic theory for flat
line bundles developed by Simpson \cite{Simpson}. As in \cite{GL2}, the idea is that
after pulling the complex $\derR \Phi_{P^\natural} \OX$ back to the universal covering space of
$\Ash$, one can use harmonic forms to
construct a linear complex that is quasi-isomorphic to the pullback in a neighborhood of a
given point. Additional technical difficulties arise however from the fact that the wedge
product of two harmonic forms is typically no longer harmonic (which was true in the 
special case of $(0,q)$-forms needed in \cite{GL2}, and led to a natural
quasi-isomorphism in that case). The new insight in this part of the
paper is that a quasi-isomorphism can still be constructed by more involved analytic methods.

Thirdly, using Theorem \ref{thm:locsyst}, we are able to derive the following generic vanishing-type property 
of the Laumon-Rothstein transform of the $\Dmod$-module $\shO_X$.

\begin{theorem}\label{thm:intro_structure}
For any smooth projective complex variety $X$ we have
$${\rm codim}_{A^\natural} {\rm Supp}~R^i \Phi_{P^\natural} \shO_X \ge 2 \big( i -
\delta(a) \big) \,\,\,\, {\rm for~all~}i.$$ 
\end{theorem}

This is best expressed in terms of the new $t$-structure alluded to above; see the equivalent formulation 
in Theorem \ref{thm:structure}. In particular, in the special case when the Albanese map is semismall, it follows
that $\derR \Phi_{P^\natural} \shO_X$ is a perverse coherent sheaf. Via formal algebraic properties of this $t$-structure,
we deduce the following vanishing statement for the higher direct images $R^i \Phi_{P^\natural} \shO_X$. It is the 
analogue of the corresponding statement for the standard Fourier-Mukai transform of $\shO_X$ conjectured by 
Green and Lazarsfeld, and proved by Hacon \cite{Hacon} (and in the K\"ahler setting in \cite{PP2}).

\begin{corollary}\label{cor:intro_vanishing}
For any smooth projective complex variety $X$ we have
$$R^i \Phi_{P^\natural} \shO_X = 0 \,\,\,\, {\rm for~all~} i <  - \delta(a).$$
\end{corollary}

Along the lines of \cite{LP}, Theorem \ref{thm:linear_intro} and Corollary \ref{cor:intro_vanishing} 
can be combined to give, by means of the BGG correspondence, a bound on the complexity of the 
cohomology algebra
\[
	P_X  := \bigoplus_{i=0}^{2n} H^{j}(X, \CC),
\]
seen as a graded module over the exterior algebra $ E := \Lambda^* H^1 (X, \CC)$ via
cup-product. We use the following conventions for the grading: in the algebra $E$, the
elements of $H^1 (X, \CC)$ have degree $-1$, while in $P_X$  the elements of $H^{j}(X, \CC)$ 
have degree $n-j$.  Our application is to bound the degrees of
the generators and syzygies of $P_X$ as an $E$-module. 

\begin{corollary}\label{cor:intro_regularity}
Let $X$ be a smooth projective complex variety of dimension $n$. Then $E$ is $\bigl(
n+\delta(a)\bigr)$-regular as a graded $E$-module.
\end{corollary}

Concretely, this means that $P_X$ is generated in degrees 
$-\delta(a), \dotsc, n$; the relations among the generators appear in degrees
$-\delta(a)-1, \dotsc, n-1$; and more generally, the $p$-th module of syzygies of $P_X$
has all its generators in degrees $-\delta(a)-p, \dotsc, n-p$. Simple examples
show that this result is optimal; see \subsecref{subsec:BGG}. 

The last part of the paper contains a few statements extending our results to more
general classes of $\Dmod$-modules on abelian varieties, not necessarily underlying 
mixed Hodge modules. The proofs are of a different nature, and will be presented 
elsewhere. We propose a few natural open problems as well.

\subsection{Acknowledgements}

We are grateful to Mark de Cataldo for pointing out the current statement of
Proposition~\ref{prop:vanishing}, which is a strengthening of our original result. 
We thank  Donu Arapura, Takuro Mochizuki and Keiji
Oguiso for useful conversations. 
M.P. is partially supported by NSF grant DMS-1101323.
C.S. is partially supported by the World Premier International Research
Center Initiative (WPI Initiative), MEXT, Japan, and by NSF grant DMS-1100606.

\section{Preliminaries}

\subsection{Mixed Hodge modules}
\label{subsec:MHM}

In this section, we recall a few aspects of Morihiko Saito's theory of mixed Hodge modules
\cite{Saito-MHM} which, as explained in the introduction, offers a convenient setting
for the results of this paper. A very good survey can also be found in
\cite{Saito-survey}. We only discuss the case of a complex algebraic variety $X$, say
of dimension $n$.

Saito defines  an abelian category $\MHM(X)$ of \emph{graded-polarizable mixed Hodge modules} on $X$. 
A mixed Hodge module is, roughly speaking, a variation of mixed Hodge structure with
singularities. Over a one-point space it is the same thing
as a graded-polarizable mixed Hodge structure with coefficients in $\QQ$.

As in classical Hodge theory, there are pure and mixed objects. We shall denote by
$\HM_{\ell}(X)$ the subcategory of polarizable (pure) Hodge modules of weight $\ell$.
On a one-point space, such an object is just a polarizable Hodge structure of weight
$\ell$. More generally, any polarizable variation of Hodge structure of weight $k$ on
a Zariski-open subset of some irreducible subvariety $Z \subseteq X$ uniquely extends
to a Hodge module of weight $\dim Z + k$ on $X$. The existence of polarizations makes
the category $\HM_{\ell}(X)$ semi-simple: each object admits a decomposition by
support, and simple objects with support equal to an irreducible subvariety $Z
\subseteq X$ are obtained from polarizable variations of Hodge structure on
Zariski-open subsets of $Z$.

When $M \in \HM_{\ell}(X)$ is pure of weight $\ell$, and we are in the presence of a
projective  morphism $f \colon X \to Y$, Saito shows that the direct image $f_* M \in
\Db \MHM(Y)$ splits non-canonically into the sum of its cohomology objects $H^i f_* M
\in \HM_{\ell+i}(Y)$. The resulting isomorphism
\[
	f_* M \simeq \bigoplus_i H^i f_* M \decal{-i}
\]
is the analogue for Hodge modules of the decomposition theorem of \cite{BBD}.
The proof of this result with methods from algebraic analysis and $\Dmod$-modules
is one of the main achievements of Saito's theory.

The precise definition of a mixed Hodge module is very involved; it uses regular holonomic
$\Dmod$-modules, perverse sheaves, and the theory of nearby and vanishing cycles.
The familiar equivalence between local systems and flat
vector bundles is replaced by the Riemann-Hilbert correspondence between perverse
sheaves and regular holonomic $\Dmod$-modules. Let $\Mmod$ be a $\Dmod$-module on
$X$; in this paper, this always means a left module over the sheaf of algebraic
differential operators $\Dmod_X$. Recall that the de Rham complex of $\Mmod$ is the
$\CC$-linear complex
\[
	\DR_X(\Mmod) = \Bigl\lbrack
		\Mmod \to \Omega_X^1 \tensor \Mmod \to \dotsb \to \Omega_X^n \tensor \Mmod
	\Bigr\rbrack,
\]
placed in degrees $-n, \dotsc, 0$. When $\Mmod$ is holonomic, $\DR_X(\Mmod)$ is
constructible and satisfies the axioms for a perverse sheaf with coefficients in
$\CC$. According to the Riemann-Hilbert correspondence of Kashiwara and Mebkhout, the
category of regular holonomic $\Dmod$-modules is equivalent to the category of
perverse sheaves $\Perv(\CC_X)$ through the functor $\Mmod \mapsto \DR_X(\Mmod)$.
Internally, a \define{mixed Hodge module} $M$ on a smooth complex algebraic variety
$X$ has three components:

\begin{enumerate}
\item A regular holonomic $\Dmod_X$-module $\Mmod$, together with a good
filtration $F_{\bullet} \Mmod$ by $\OX$-coherent subsheaves such that $\Gr_{\bullet}^F \!
\Mmod$ is coherent over $\Gr_{\bullet}^F \Dmod_X$. This filtration plays the role of
a Hodge filtration on $\Mmod$.
\item A perverse sheaf $K \in \Perv(\QQ_X)$, together with an isomorphism
\[
	\alpha \colon \DR_X(\Mmod) \to K \tensor_{\QQ} \CC.
\]
This isomorphism plays the role of a $\QQ$-structure on the $\Dmod$-module
$\Mmod$.  (When $M$ is a mixed Hodge structure, $K$ is the
underlying $\QQ$-vector space; when $M$ corresponds to a polarizable variation of Hodge
structure $H$ on a Zariski-open subset of $Z \subseteq X$, $K$ is the intersection complex 
$\IC_Z(H)$ of the underlying local system.)
\item  A finite increasing \define{weight filtration} $W_{\bullet} M$ of $M$ by
objects of the same kind, compatible with $\alpha$, such that the graded quotients
$\Gr_{\ell}^W M = W_{\ell} M / W_{\ell-1} M$ belong to $\HM_{\ell}(X)$. 
\end{enumerate}

These components are subject to several conditions, which are defined by induction on the
dimension of the support of $\Mmod$. On a one-point space, a mixed Hodge module is a
graded-polarizable mixed Hodge structure; in general, Saito's conditions require that
the nearby and vanishing cycles of $\Mmod$ with respect to any locally defined
holomorphic function are again mixed Hodge modules (now on a variety of dimension
$n - 1$); the existence of polarizations; etc. An important fact is that these local conditions 
are preserved after taking direct images, and that they are sufficient to obtain global results such as the
decomposition theorem.

The most basic Hodge module on a smooth variety $X$ is the trivial Hodge module,
denoted by $\QHX \decal{n}$. Its underlying $\Dmod$-module is the structure sheaf
$\OX$, with the trivial filtration $F_0 \OX = \OX$; its underlying perverse sheaf is
$\QQ_X \decal{n}$, the constant local system placed in degree $-n$. The comparison
isomorphism is nothing but the well-known fact that the usual holomorphic de Rham complex
\[
	\OX \to \Omega_X^1 \to \Omega_X^2 \to \dotsb \to \Omega_X^n
\]
is a resolution of the constant local system $\CC_X$. 
In general, all admissible variations of mixed Hodge 
structures are mixed Hodge modules, and any mixed Hodge module is generically one such.
The usual cohomology groups of $X$ can be obtained
as the cohomology of the complex $p_* p^* \QQ^H \in \Db \MHM(\pt)$,
where $p \colon X \to \pt$ denotes the map to a one-point space, and $\QQ^H$ is the
Hodge structure of weight zero on $\QQ$. There are similar expressions for cohomology
with compact supports, intersection cohomology, and so forth, leading to a uniform
description of the mixed Hodge structures on all of these groups.

\subsection{Three theorems about mixed Hodge modules}
\label{subsec:three}

In this section, we recall from the literature three useful theorems about the associated graded object
$\Gr_{\bullet}^F \Mmod$, for a filtered $\Dmod$-module $(\Mmod, F)$ underlying a
mixed Hodge module.

During the discussion, $X$ will be a smooth complex projective variety of dimension
$n$, and $M \in \MHM(X)$ a mixed Hodge module on $X$. As usual, we denote the
underlying filtered $\Dmod$-module by $(\Mmod, F)$. Recall that the associated graded
of the sheaf of differential operators $\Dmod_X$, with respect to the filtration by
order of differential operators, is isomorphic to $\shA_X^{\bullet} = \Sym^{\bullet}
\shT_X$, the symmetric algebra of the tangent sheaf of $X$. Since $\Gr_{\bullet}^F
\Mmod$ is finitely generated over this sheaf of algebras, it defines a coherent sheaf
$\shC(\Mmod, F)$ on the cotangent bundle $T^{\ast} X$.  The support of this sheaf is
the characteristic variety of the $\Dmod$-module $\Mmod$, and therefore of pure
dimension $n$ because $\Mmod$ is holonomic. 

The first of the three theorems is Saito's generalization of the Kodaira vanishing
theorem. Before we state it, observe that the filtration $F_{\bullet} \Mmod$ is
compatible with the $\Dmod$-module structure on $\Mmod$, and therefore induces a 
filtration on the de Rham complex of $\Mmod$ by the formula
\begin{equation} \label{eq:filtration}
	F_k \DR_X(\Mmod) = \Bigl\lbrack
		F_k \Mmod \to \Omega_X^1 \tensor F_{k+1} \Mmod \to \dotsb 
			\to \Omega_X^n \tensor F_{k+n} \Mmod
	\Bigr\rbrack.
\end{equation}
The associated graded complex for the filtration in \eqref{eq:filtration} is 
\[
	\Gr_k^F \DR_X(\Mmod) = \Big\lbrack
		\Gr_k^F \Mmod \to \OmX{1} \tensor \Gr_{k+1}^F \Mmod \to \dotsb \to
			\OmX{n} \tensor \Gr_{k+n}^F \Mmod
	\Big\rbrack,
\]
which is now a complex of coherent sheaves of $\OX$-modules in degrees $-n, \dotsc,
0$. This complex satisfies the following Kodaira-type vanishing theorem.

\begin{theorem}[Saito]
\label{saito_vanishing}
Let $(\Mmod, F)$ be the filtered $\Dmod$-module underlying a mixed Hodge
module on a smooth projective variety $X$, and let $L$ be any ample line bundle.
\begin{enumerate}
\item One has $\mathbf{H}^i \bigl( X, \Gr_k^F \DR_X(\Mmod) \tensor L \bigr) = 0$ for
all $i > 0$.
\item One has $\mathbf{H}^i \bigl( X, \Gr_k^F \DR_X(\Mmod) \tensor L^{-1} \bigr) = 0$ for
all $i < 0$.
\end{enumerate}
\end{theorem}

\begin{proof}
The proof works by reducing the assertion to a vanishing theorem for perverse
sheaves on affine varieties, with the help of Saito's formalism. Details can be
found in \cite{Saito-MHM}*{Proposition~2.33}.
\end{proof}

\begin{example}
For the trivial Hodge module $M = \QHX \decal{n}$ on $X$, we have $\Mmod = \OX$, and
therefore $\Gr_{-n}^F \DR_X(\OX) = \omega_X$. This shows that
Theorem~\ref{saito_vanishing} generalizes the Kodaira vanishing theorem. 
Since $\Gr_{-p}^F \DR_X(\OX) = \Omega_X^p \decal{n-p}$, it also generalizes the
Nakano vanishing theorem.
\end{example}

The second theorem gives more information about the coherent sheaf $\shC(\Mmod, F)$
on the cotangent bundle of $X$. Before stating it, we recall the definition of the
\emph{Verdier dual} $M' = \DX M$ of a mixed Hodge module. If $\rat M$ is the perverse
sheaf underlying $M$, then $\rat M'$ is simply the usual topological Verdier dual. On
the level of $\Dmod$-modules, note that since $\Mmod$ is left $\Dmod$-module, the
dual complex
\[
	\derR \shHom_{\Dmod_X} \bigl( \Mmod, \Dmod_X \decal{n} \bigr)
\]
is naturally a complex of right $\Dmod$-modules; since $\Mmod$ is holonomic, it is 
quasi-isomorphic to a single right $\Dmod$-module. If $\Mmod'$ denotes the left
$\Dmod$-module underlying the Verdier dual $M'$, then that right $\Dmod$-module is
$\omX \tensor_{\OX} \Mmod'$. Thus we have
\[
	\derR \shHom_{\Dmod_X} \bigl( \Mmod, \Dmod_X \decal{n} \bigr) 
		\simeq \omX \tensor_{\OX} \Mmod',
\]
where the right $\Dmod$-module structure on $\omX \tensor \Mmod'$ is given by the
rule $\xi \cdot (\omega \tensor m') = (\xi \omega) \tensor m' - \omega \tensor
(\xi m')$ for $\xi \in \shT_X$. The Hodge filtration on $\Mmod$ induces a
filtration on the dual complex and hence on $\Mmod'$. The following result shows that
this induced filtration is well-behaved, in a way that makes duality and passage to
the associated graded compatible with each other.

\begin{theorem}[Saito]
\label{thm:Verdier}
Let $M$ be a mixed Hodge module on a smooth complex algebraic variety $X$ of
dimension $n$, and let $M'$ denote its Verdier dual. Then
\[
	\derR \shHom_{\shA_X^{\bullet}} \bigl( \Gr_{\bullet}^F \Mmod, 
		\shA_X^{\bullet} \decal{n} \bigr)  \simeq 
		\omega_X \tensor_{\OX} \Gr_{\bullet + 2n}^F \Mmod',
\]
where sections of $\Sym^k \shT_X$ act with an extra factor of $(-1)^k$ on
the right-hand side. If we consider both sides as coherent sheaves on $T^{\ast} X$,
we obtain
\[
	\derR \shHom \bigl( \shC(\Mmod, F), \shO_{T^{\ast} X} \decal{n} \bigr) \simeq  
		\pu \omega_X \tensor (-1)_{T^{\ast} X}^{\ast} \shC(\Mmod', F),
\]
where $p \colon T^{\ast} X \to X$ is the projection. In particular, $\shC(\Mmod, F)$
is always a Cohen-Macaulay sheaf of dimension $n$.
\end{theorem}

\begin{proof}
That the coherent sheaf $\shC(\Mmod, F)$ on $T^{\ast} X$ is Cohen-Macaulay is proved
in \cite{Saito-HM}*{Lemme 5.1.13}.  For an explanation of how this fact implies the
formula for the dual of the graded module $\Gr_{\bullet}^F \Mmod$, see
\cite{duality}*{\S3.1}. 
\end{proof}

The last of the three theorems gives a formula for the associated graded object of
$f_* M$, where $f \colon X \to Y$ is a projective morphism between two smooth complex
algebraic varieties. The formula itself first appears in a paper by Laumon
\cite{Laumon}*{Construction~2.3.2}, but it is not true in the generality claimed there
(that is to say, for arbitrary filtered $\Dmod$-modules).

Before stating the precise result, we give an informal version. The following diagram
of morphisms, induced by $f$, will be used throughout:
\begin{diagram}{3em}{2.5em}
\matrix[math] (m) { 
T^{\ast} X & T^{\ast} Y \times_Y X & X \\
& T^{\ast} Y & Y \\ }; 
\path[to] (m-1-2) edge node[above] {$\mathit{df}$} (m-1-1);
\path[to] (m-1-2) edge node[auto] {$p_1$} (m-2-2)
		edge node[auto] {$p_2$} (m-1-3);
\path[to] (m-2-2) edge node[auto] {$p$} (m-2-3);
\path[to] (m-1-3) edge node[auto] {$f$} (m-2-3);
\end{diagram}
Let $M \in \MHM(X)$ be a mixed Hodge module, let $(\Mmod, F)$ be its underlying
filtered $\Dmod$-module, and write $\shC(\Mmod, F)$ for the associated coherent sheaf on
$T^{\ast} X$. We shall denote the complex of filtered holonomic $\Dmod$-modules
underlying the direct image $\fl M \in \Db \MHM(Y)$ by the symbol $\fl(\Mmod, F)$.
Then \define{Laumon's formula} claims that
\[
	\shC \bigl( \fl(\Mmod, F) \bigr) \simeq 
		\derR p_{1\ast} \bigl( \derL \mathit{df}^{\ast} \shC(\Mmod, F) 
			\tensor p_2^{\ast} \omega_{X/Y} \bigr)
\]
as objects in the derived category $\Dbcoh(\shO_{T^{\ast} Y})$.
Unfortunately, the passage to coherent sheaves on the cotangent bundle loses the
information about the grading; it is therefore better to work directly with
the graded modules. As above, we write $p \colon T^{\ast} X \to X$ for the
projection, and denote the graded sheaf of $\OX$-algebras $p_* \shO_{T^{\ast} X} =
\Sym \shT_X$ by the symbol $\AX$.

\begin{theorem}[Laumon]
\label{thm:laumon}
Let $f \colon X \to Y$ be a projective morphism between smooth complex algebraic
varieties, and let $M \in \MHM(X)$. Then with notation as above,
\[
	\Gr_{\bullet}^F \fl(\Mmod, F) \simeq
		\derR \fl \Bigl( \omega_{X/Y} \tensor_{\OX}
		\Gr_{\bullet + \dim X - \dim Y}^F \Mmod \tensor_{\AX} \fu \AY \Bigr).
\]
\end{theorem}

\begin{proof}
This can be easily proved using Saito's formalism of induced $\Dmod$-modules;
both the factor of $\omega_{X/Y}$ and the shift in the grading come from the
transformation between left and right $\Dmod$-modules that is involved. To illustrate
what is going on, we shall outline a proof based on factoring $f$ through its graph.
By this device, it suffices to verify the formula in two cases: (1) for a regular
closed embedding $f \colon X \into Y$; (2) for a projection $f \colon Y \times
Z \to Y$ with $Z$ smooth and projective.

We first consider the case where $f \colon X \into Y$ is a regular
closed embedding of codimension $r$. Here $\omega_{X/Y} \simeq \det N_{X \mid Y}$.
Working locally, we may assume without loss of generality that $Y = X \times \CC^r$.
Let $t_1, \dotsc, t_r$ denote the coordinates on $\CC^r$, and let $\partial_1,
\dotsc, \partial_r$ be the corresponding vector fields. Then 
\[
	\fu \AY = \AX \lbrack \partial_1, \dotsc, \partial_r \rbrack,
\]
and so the formula we need to prove is that
\[
	\Gr_{\bullet}^F \fl(\Mmod, F) \simeq \fl \Bigl( \Gr_{\bullet-r}^F \Mmod
		\tensor_{\AX} \AX \lbrack \partial_1, \dotsc, \partial_r \rbrack \Bigr).
\]
By \cite{Saito-HM}*{p.~850}, we have $\fl(\Mmod, F) \simeq \Mmod \lbrack \partial_1,
\dotsc, \partial_r \rbrack$, with filtration given by
\[
	F_p \, \fl(\Mmod, F)
		\simeq \sum_{k + \abs{\nu} \leq p - r} 
		\fl \bigl( F_k \Mmod \bigr) \tensor \partial^{\nu}.
\]
Consequently, we get
\[
	\Gr_p^F \fl(\Mmod, F) 
		\simeq \bigoplus_{\nu \in \NN^r} \fl \bigl( \Gr_{p-r-\abs{\nu}}^F \Mmod \bigr) 
			\tensor \partial^{\nu},
\]
which is the desired formula. Globally, the factor of $\det N_{X \mid Y}$ is needed
to make the above isomorphism coordinate independent.

Next, consider the case where $X = Y \times Z$, with $Z$ smooth and projective of
dimension $r$, and $f = p_1$. Then $\omega_{X/Y} \simeq p_2^{\ast} \omega_Z$. In this
case, we have
\[
	\fl(\Mmod, F) = \derR p_{1\ast} \DR_{Y \times Z / Y}(\Mmod),
\]
where $\DR_{Y \times Z / Y}(\Mmod)$ is the relative de Rham complex
\[
	\Bigl\lbrack \Mmod \to \Omega_{Y \times Z / Y}^1 \tensor \Mmod \to \dotsb
		\to \Omega_{Y \times Z / Y}^r \tensor \Mmod \Bigr\rbrack,
\]
supported in degrees $-r, \dotsc, 0$. As in \eqref{eq:filtration}, the Hodge filtration on
the $\Dmod$-module $\Mmod$ induces a filtration on the relative de Rham complex by
the formula
\[
	F_k \DR_{Y \times Z / Y}(\Mmod) = \Bigl\lbrack 
		F_k \Mmod \to \Omega_{Y \times Z / Y}^1 \tensor F_{k+1} \Mmod \to \dotsb
		\to \Omega_{Y \times Z / Y}^r \tensor F_{k+r} \Mmod 
	\Bigr\rbrack.
\]
Now the key point is that since $Z$ is smooth and projective, the induced filtration
on the direct image complex is strict by \cite{Saito-MHM}*{Theorem~2.14}. It follows that
\[
	\Gr_{\bullet}^F \fl(\Mmod, F) \simeq \derR p_{1\ast} 
		\Bigl\lbrack \Gr_{\bullet}^F \Mmod \to p_2^{\ast} \Omega_Z^1 \tensor \Gr_{\bullet+1}^F \Mmod 
		\to \dotsb \to p_2^{\ast} \Omega_Z^r \tensor \Gr_{\bullet+r}^F \Mmod \Bigr\rbrack.
\]
On the other hand, a graded locally free resolution of $\omega_Z$ as a graded
$\AZ$-module is given by the complex
\[
	\Bigl\lbrack \AZ^{\bullet-r} \to \Omega_Z^1 \tensor \AZ^{\bullet-r+1} \to 
		\dotsb \to \Omega_Z^r \tensor \AZ^{\bullet} \Bigr\rbrack,
\]	
again in degrees $-r, \dotsc, 0$. Therefore $\omega_{X/Y} \tensor_{\OX} \fu \AY$ is
naturally resolved, as a graded $\AX$-module, by the complex
\[
	\Bigl\lbrack \AX^{\bullet-r} \to p_2^{\ast} \Omega_Z^1 \tensor \AX^{\bullet-r+1} \to 
		\dotsb \to p_2^{\ast} \Omega_Z^r \tensor \AX^{\bullet} \Bigr\rbrack,
\]
and so $\Gr_{\bullet+r}^F \Mmod \tensor_{\AX} \fu \AY \tensor_{\OX} \omega_{X/Y}$ is
represented by the complex
\[
	\Bigl\lbrack \Gr_{\bullet}^F \Mmod \to p_2^{\ast} \Omega_Z^1 \tensor \Gr_{\bullet+1}^F \Mmod 
		\to \dotsb \to p_2^{\ast} \Omega_Z^r \tensor \Gr_{\bullet+r}^F \Mmod \Bigr\rbrack,
\]
from which the desired formula follows immediately.
\end{proof}

\subsection{Perverse coherent sheaves}
\label{subsec:t-structure}

Let $X$ be a smooth complex algebraic variety. In this section, we record some
information about perverse $t$-structures on $\Dbcoh(\OX)$, emphasizing two that are
of interest to us here. We follow the notation introduced by Kashiwara
\cite{Kashiwara}. For a (possibly non-closed) point $x$ of the scheme $X$, we write
$\kappa(x)$ for the residue field at the point, $i_x \colon \Spec \kappa(x) \into X$
for the inclusion, and $\codim(x) = \dim \shO_{X,x}$ for the codimension of the
closed subvariety $\overline{ \{x\} }$.

A \define{supporting function} on $X$ is a function $p \colon X \to \ZZ$ from the
topological space of the scheme $X$ to the integers, with the property that $p(y)
\geq p(x)$ whenever $y \in \overline{ \{x\} }$. Given such a supporting function, one
defines two families of subcategories
\[
	\pDtcoh{p}{\leq k}(\OX) = 
		\menge{M \in \Dbcoh(\OX)}%
			{\text{$\derL i_x^{\ast} M \in \Dtcoh{\leq k+p(x)} %
			\bigl( \kappa(x) \bigr)$ for all $x \in X$}}
\]
and
\[
	\pDtcoh{p}{\geq k}(\OX) = 
		\menge{M \in \Dbcoh(\OX)}%
			{\text{$\derR i_x^! M \in \Dtcoh{\geq k+p(x)} %
			\bigl( \kappa(x) \bigr)$ for all $x \in X$}}.
\]
The following fundamental result is proved in \cite{Kashiwara}*{Theorem~5.9} and,
based on an idea of Deligne, in \cite{AB}*{Theorem 3.10}.

\begin{theorem} \label{thm:kashiwara}
The above subcategories define a bounded $t$-structure on $\Dbcoh(\OX)$ iff the supporting
function has the property that $p(y) - p(x) \leq \codim(y) - \codim(x)$ for every
pair of points $x,y \in X$ with $y \in \overline{ \{x\} }$.
\end{theorem}

For example, $p=0$ corresponds to the standard $t$-structure on
$\Dbcoh(\OX)$. An equivalent way of putting the condition in
Theorem~\ref{thm:kashiwara} is that the dual function $\hat{p}(x) = \codim(x) - p(x)$
should again be a supporting function. If that is the case, one has the identities
\begin{align*}
	\pDtcoh{\hat{p}}{\leq k}(\OX) &=
		 \derR \shHom \Bigl( \pDtcoh{p}{\geq -k}(\OX), \OX \Bigr) \\
	\pDtcoh{\hat{p}}{\geq k}(\OX) &=
		 \derR \shHom \Bigl( \pDtcoh{p}{\leq -k}(\OX), \OX \Bigr),
\end{align*}
which means that the duality functor $\derR \shHom(\argbl, \OX)$ exchanges the
two perverse $t$-structures defined by $p$ and $\hat{p}$. The heart of the
$t$-structure defined by $p$ is denoted
\[
	\pCoh{p}(\OX) = \pDtcoh{p}{\leq 0}(\OX) \cap \pDtcoh{p}{\geq 0}(\OX),
\]
and is called the abelian category of \define{$p$-perverse coherent sheaves}. 

We are interested in two special cases of Kashiwara's result. One is that the
set of objects $E \in \Dbcoh(\OX)$ with $\codim \Supp H^i(E) \geq i$ for all $i$ forms
part of a $t$-structure (the ``dual standard $t$-structure'' in Kashiwara's
terminology); the other is that the same is true for the set of objects with
$\codim \Supp H^i(E) \geq 2i$ for all $i$. This can be formalized in the following way.

We first define a supporting function $c \colon X \to \ZZ$ by the formula $c(x) =
\codim(x)$. Since the dual function is $\hat{c} = 0$, it is clear that $c$ defines a
$t$-structure on $\Dbcoh(\OX)$, namely the dual of the standard $t$-structure. 
The following result by Kashiwara \cite{Kashiwara}*{Proposition~4.3} gives an alternative
description of $\cCoh(\OX)$.

\begin{lemma} \label{lem:perverse}
The following three conditions on $E \in \Dbcoh(\OX)$ are equivalent:
\begin{enumerate}
\item $E$ belongs to $\cCoh(\OX)$.
\item The dual object $\derR \shHom(E, \OX)$ is a coherent sheaf.
\item $E$ satisfies $\codim \Supp H^i(E) \geq i$ for every $i \geq 0$ (equivalent to $E \in \cDtcoh{\leq k}(\OX)$)
and it is quasi-isomorphic to a bounded complex of flat $\OX$-modules in
non-negative degrees (equivalent to $E \in \cDtcoh{\geq k}(\OX)$).
\end{enumerate}
\end{lemma}

We now define a second function $m \colon X \to \ZZ$ by the formula
\[
	m(x) = \left\lfloor \tfrac{1}{2} \codim(x) \right\rfloor.
\]
It is easily verified that both $m$ and the dual function
\[
	\hat{m}(x) = \left\lceil \tfrac{1}{2} \codim(x) \right\rceil
\]
are supporting functions. As a consequence of Theorem~\ref{thm:kashiwara}, $m$ defines a
bounded $t$-structure on $\Dbcoh(\OX)$; objects of the heart $\mCoh(\OX)$ will be
called $m$-perverse coherent sheaves. (We use this letter because $m$ and $\hat{m}$
are as close as one can get to ``middle perversity''. Since the equality $m =
\hat{m}$ can never be satisfied, there is of course no actual middle perversity for
coherent sheaves.) 

The next lemma follows easily from \cite{Kashiwara}*{Lemma~5.5}.
\begin{lemma} \label{lem:m-structure}
The perverse $t$-structures defined by $m$ and $\hat{m}$ satisfy
\begin{align*}
	\mDtcoh{\leq k}(\OX) &= \menge{E \in \Dbcoh(X)}%
		{\text{$\codim \Supp H^i(E) \geq 2(i-k)$ for all $i \in \ZZ$}} \\
	\pDtcoh{\hat{m}}{\leq k}(\OX) &= \menge{E \in \Dbcoh(X)}%
		{\text{$\codim \Supp H^i(E) \geq 2(i-k)-1$ for all $i \in \ZZ$}}.
\end{align*}
By duality, this also describes the subcategories with $\geq k$.
\end{lemma}

Consequently, an object $E \in \Dbcoh(\OX)$ is an $m$-perverse coherent sheaf iff
\[
	\codim \Supp H^i(E) \geq 2i \quad \text{and} \quad
		\codim \Supp R^i \shHom(E, \OX) \geq 2i-1
\]
for every $i \in \ZZ$. This shows one more time that the category of $m$-perverse
coherent sheaves is not preserved by the duality functor $\derR \shHom(\argbl, \OX)$.

\begin{lemma} \label{lem:heart}
For any integer $k \ge 0$, if $E \in \mDtcoh{\leq k}(\OX)$, then $E \in \Dtcoh{\geq -k}(\OX)$.
\end{lemma}

\begin{proof}
By translation it is enough to prove this for $k =0$, where the result is obvious
from the fact that $m(x) \geq 0$. 
\end{proof}

\subsection{Integral functors and GV-objects} 

Let $X$ and $Y$ be smooth projective complex varieties of dimensions $n$ and $g$
respectively, and let $P\in \Dbcoh(\shO_{X\times Y})$ be an object inducing integral
functors 
\[
	\derR \Phi_P \colon \Dbcoh (\OX) \to \Dbcoh(\OY),
~~ \derR \Phi(\argbl) := \derR {p_Y}_* ( p_X^*(\argbl) \overset{\derL}{\otimes} P)
\]
and
$$\derR \Psi_P \colon \Dbcoh(\OY) \to \Dbcoh (\OX),
~~\derR \Psi(\argbl) := \derR {p_X}_* ( p_Y^*(\argbl) \overset{\derL}{\otimes} P).$$

Let $E \in \Dbcoh(\OX)$. There is a useful cohomological criterion for checking whether 
the integral transform $\FM E$ is a perverse coherent sheaf on $Y$ with respect to the 
$t$-structure defined by $c$ from the previous subsection. Note first that by base change, for any 
sufficiently positive ample line bundle $L$ on
$Y$, the transform $\FMi(L^{-1})$ is supported in degree $g$, and $R^g
\Psi_P(L^{-1})$ is a locally free sheaf on $A$. 
The following is contained in \cite{Popa}*{Theorem~3.8 and 4.1}, and is 
partly based on the previous \cite{Hacon}*{Theorem 1.2} and \cite{PP}*{Theorem A}.

\begin{theorem}\label{perverse_criterion}
For $k \in \NN$, the following are equivalent:

\noindent
(1) $ \derR\Phi_P E$ belongs to $\cDtcoh{\leq k}(\shO_Y)$. 

\noindent
(2) For any sufficiently ample line bundle $L$ on $Y$, and every $i > k$,
\[
	H^i \bigl( X, E \tensor R^g \Psi_P(L^{-1}) \bigr) = 0.
\]

\noindent
(3) $R^i \Phi_{P^\vee} (\omega_X \otimes E^\vee) = 0$ for all $i < n - k$.

Moreover, when $k = 0$, if $E$ is a sheaf---or more generally a \emph{geometric} $\GV$-object, in the
language of \cite{Popa}*{Definition 3.7}---satisfying the equivalent conditions
above, then $ \derR \Phi_P E$  is a perverse sheaf in $\cCoh(\OY)$.
\end{theorem}

\begin{definition}
An object $E$ satisfying the equivalent conditions in the Theorem with $k = 0$, is called a
\define{$\GV$-object} (with respect to $P$); if it is moreover a sheaf, then it is
called a \define{$\GV$-sheaf}.
\end{definition}

\section{Mixed Hodge modules and generic vanishing}

\subsection{Mixed Hodge modules on abelian varieties}

Let $A$ be a complex abelian variety of dimension $g$, and let $M \in \MHM(A)$ be a
mixed Hodge module on $A$. As usual, we denote the underlying filtered holonomic
$\Dmod$-module by $(\Mmod, F)$. From Theorem~\ref{saito_vanishing}, 
we know that the associated graded pieces of the de Rham complex
$\DR(\Mmod, F)$ satisfy an  analogue of Kodaira vanishing. A key observation is that
on abelian varieties, the same vanishing theorem holds for the individual coherent
sheaves $\Gr_k^F \Mmod$, due to the fact that the cotangent bundle of $A$ is trivial.
As we shall see, this implies that each $\Gr_k^F \Mmod$ is a $\GV$-sheaf on $A$ (and
therefore transforms to a perverse coherent sheaf on $\Ah$ with respect to the 
$t$-structure given by $c$ in \S\ref{subsec:t-structure}).

\begin{lemma} \label{lem:kodaira}
Let $(\Mmod, F)$ be the filtered $\Dmod$-module underlying a mixed Hodge module on
$A$, and let $L$ be an ample line bundle. Then for each $k \in \ZZ$, we have 
\[
        H^i \bigl( A, \Gr_k^F \Mmod \tensor L \bigr) = 0
\]
for every $i > 0$.
\end{lemma}

\begin{proof}
Consider for each $k \in \ZZ$ the complex of coherent sheaves
\[
	\Gr_k^F \DR_A(\Mmod) = \Big\lbrack
		\Gr_k^F \Mmod \to \OmA{1} \tensor \Gr_{k+1}^F \Mmod \to \dotsb \to
			\OmA{g} \tensor \Gr_{k+g}^F \Mmod
	\Big\rbrack,
\]
supported in degrees $-g, \dotsc, 0$. According to Theorem~\ref{saito_vanishing},
this complex has the property that, for $i > 0$, 
\[
	\mathbf{H}^i \bigl( A, \Gr_k^F \DR_A(\Mmod) \tensor L \bigr) = 0.
\]
Using the fact that $\OmA{1} \simeq \OA^{\oplus g}$, one can deduce the asserted
vanishing theorem for the individual sheaves $\Gr_k^F \Mmod$ by induction on $k$. Indeed, since
$\Gr_k^F \Mmod = 0$ for $k \ll 0$, inductively one has for each $k$ a distinguished
triangle 
\[	
	E_k \to \Gr_k^F \DR_A(\Mmod) \to \Gr_{k+g}^F \Mmod \to E_k[1],
\]
with $E_k$ an object satisfying $\mathbf{H}^i ( A, E_k \tensor L) = 0$.
\end{proof}

We now obtain the first theorem of the introduction, by combining
Lemma~\ref{lem:kodaira} with Theorem \ref{perverse_criterion} and a trick invented by
Mukai. As mentioned above, the method is the same as in Hacon's proof of the generic
vanishing theorem \cites{Hacon,PP}.


\begin{proof}[Proof of Theorem \ref{hm_abelian}]
Let $L$ be an ample line bundle on $\Ah$. By Theorem \ref{perverse_criterion},
it suffices to show that
\[
	H^i \bigl( A, \Gr_k^F \Mmod \tensor R^g \Psi_P(L^{-1}) \bigr) = 0
\]
for $i > 0$. Let $\varphi_L \colon \Ah \to A$ be
the isogeny induced by $L$. Then, by virtue of $\varphi_L$ being \'etale,
\[
	H^i \bigl( A, \Gr_k^F \Mmod \tensor R^g \Psi_P(L^{-1}) \bigr) \into
	H^i \bigl( \Ah, \varphi_L^{\ast} \Gr_k^F \Mmod \tensor 
		\varphi_L^{\ast} R^g \Psi_P(L^{-1}) \bigr)
\]
is injective, and so we are reduced to proving that the group on the right vanishes
whenever $i > 0$.

Let $N = \varphi_L^{\ast} M$ be the pullback of the mixed Hodge module $M$ to $\Ah$.
If $(\Nmod, F)$ denotes the underlying filtered holonomic $\Dmod$-module, then 
$F_k \Nmod = \varphi_L^{\ast} F_k \Mmod$ because $\varphi_L$ is \'etale. On
the other hand, by \cite{Mukai} 3.11
\[
	\varphi_L^{\ast} R^g \Psi_P(L^{-1}) \simeq H^0(\Ah, L) \tensor L.
\]
We therefore get
\[
	H^i \bigl( \Ah, \varphi_L^{\ast} \Gr_k^F \Mmod \tensor 
		\varphi_L^{\ast} R^g \Psi_P(L^{-1}) \bigr)
	\simeq H^0(\Ah, L) \tensor H^i \bigl( \Ah, \Gr_k^F \Nmod \tensor L \bigr),
\]
which vanishes for $i > 0$ by Lemma~\ref{lem:kodaira}.
\end{proof}

Going back to the de Rham complex, it is worth recording the complete information one
can obtain about $\Gr_k^F \DR_A(\Mmod)$ from Saito's theorem, since this produces
further natural examples of $\GV$-objects on $A$.  One one hand, just as in the proof
of Theorem \ref{hm_abelian}, we see that each $\Gr_k^F \DR_A(\Mmod)$ is a
$\GV$-object. On the other hand, since $\Gr_k^F \DR_A(\Mmod)$ is supported in
non-positive degrees, its transform with respect to the standard 
Fourier-Mukai functor
$$\derR \Phi_P \colon \Dbcoh (\OA) \to \Dbcoh(\shO_{\widehat{A}}),$$
given by the normalized Poincar\'e bundle on $A\times \widehat{A}$, could a priori have cohomology in
negative degrees. The following proposition shows that this is not the case.

\begin{proposition}
If $(\Mmod, F)$ underlies a mixed Hodge module on $A$, then 
\[
	\FM \bigl( \Gr_k^F \DR_A(\Mmod) \bigr) \in \Dtcoh{\geq 0}(\Ah).
\]
\end{proposition}

\begin{proof}
By a standard application of Serre vanishing  (see \cite{PP}*{Lemma 2.5}),  it suffices to show that for any sufficiently 
positive ample line bundle $L$ on $\widehat{A}$,
\[
	\mathbf{H}^i \bigl( A,  \Gr_k^F \DR(\Mmod, F) \tensor R^0 \Psi_P(L) \bigr) = 0
\]
for $i < 0$. Assuming that $L$ is symmetric, $R^0 \Psi_P(L)$ is easily seen to be the dual of the locally free
sheaf $R^g \Psi_P(L^{-1})$, and so arguing as in the proof of Theorem \ref{hm_abelian}, we are reduced to proving that
\[
	\mathbf{H}^i \bigl( \Ah, \varphi_L^{\ast}  \Gr_k^F \DR(\Mmod, F) \tensor L^{-1} \bigr) = 0
\]
whenever $i < 0$. But since $\varphi_L^{\ast}  \Gr_k^F \DR(\Mmod, F) \simeq \Gr_k^F \DR(\Nmod, F)$, this is
an immediate consequence of part (2) of Saito's vanishing Theorem \ref{saito_vanishing}.
\end{proof}

\begin{corollary}
If $(\Mmod, F)$ underlies a mixed Hodge module on $A$, then 
$$\Gr_k^F \DR_A(\Mmod) {\rm ~and~} \derR \shHom \bigl( \Gr_k^F \DR_A(\Mmod), \OA
\decal{g} \bigr)$$ are $\GV$-objects on $A$. Therefore the graded pieces of the de
Rham complexes associated to such $\Dmod$-modules form a class of $\GV$-objects which
is closed under Grothendieck duality.
\end{corollary}
\begin{proof}
Note that, by definition, the $\GV$-objects on $A$ are precisely those $E \in
\Dbcoh(\OA)$ for which $\derR \Phi_P E \in \cDtcoh{\le 0} (\OAh) $, and that
the Fourier-Mukai and duality functors satisfy the exchange formula
$$(\derR\Phi_P E)^\vee \simeq \derR{\Phi}_{P^{-1}} (E^\vee) [g]$$
for any object $E$ in $\Dbcoh(\OA)$ (see e.g.\@ \cite{PP}*{Lemma 2.2}). Observing that
$P^{-1} \simeq (1\times(-1))^* P$, the statement follows from
the Proposition above and the equivalence in \cite{Kashiwara}*{Proposition~4.3},
which says that $\Dtcoh{\geq 0}(\OAh)$ is the category obtained by applying
the functor $\derR \shHom(\argbl, \OAh)$ to $\cDtcoh{\le 0}(\OAh)$ .
\end{proof}

\subsection{The decomposition theorem for the Albanese map}
\label{subsec:albanese}

Let $X$ be a smooth complex projective variety of dimension $n$, let $A = \Alb(X)$ be
its Albanese variety, and let $a \colon X \to A$ be the Albanese map (for some choice
of base point). As before, we set $g = \dim A$. 

To understand the Hodge theory of the Albanese map, we consider the direct image $a_*
\QHX \decal{n}$  in $\Db \MHM(A)$. Here $\QHX \decal{n}$ is the trivial Hodge module
on $X$; its underlying perverse sheaf is $\QQ_X \decal{n}$, and its underlying
filtered holonomic $\Dmod$-module is $(\OX, F)$, where $\Gr_k^F \OX = 0$ for $k \neq
0$. Note that $\QHX \decal{n} \in \HM_n(X)$ is pure of weight $n$, because $\QQ_X$ is
a variation of Hodge structure of weight $0$. According to the decomposition theorem
\cite{Saito-HM}*{Th\'eor\`eme~5.3.1 and Corollaire~5.4.8}, we have a (non-canonical)
decomposition
\[
	a_* \QHX \decal{n} \simeq \bigoplus_{i \in \ZZ} M_i \decal{-i},
\]
where each $M_i = H^i a_* \QHX \decal{n}$ is a pure Hodge module on $A$ of weight
$n+i$. Each $M_i$ can be further decomposed (canonically) into a finite sum of simple
Hodge modules
\[
	M_i = \bigoplus_j M_{i, j},
\]
where $M_{i,j}$ has strict support equal to some irreducible subvariety $Z_{i,j}
\subseteq A$; the perverse sheaf underlying $M_{i,j}$ is the intersection complex of a
local system on a Zariski-open subset of $Z_{i,j}$. Note that since $a$
is projective, we have the Lefschetz isomorphism \cite{Saito-MHM}*{Th\'eor\`eme~1}
\[
	M_{-i} \simeq M_i(i),
\]
induced by $i$-fold cup product with the first Chern class of an ample line bundle.
The Tate twist, necessary to change the weight of $M_i$ from $n+i$ to $n-i$, requires
some explanation. If $(\Mmod_i, F)$ denotes the filtered $\Dmod$-module underlying
$M_i$, then the filtered $\Dmod$-module underlying $M_i(i)$ is $(\Mmod_i,
F_{\bullet-i})$; thus the above isomorphism means that $F_k \Mmod_{-i} \simeq F_{k-i}
\Mmod_i$.

To relate this decomposition with some concrete information about the Albanese map,
we use Laumon's formula (Theorem~\ref{thm:laumon}) to compute the associated graded
of the complex of filtered $\Dmod$-modules $a_*(\OX, F)$ underlying $a_* \QHX
\decal{n} \in \Db \MHM(A)$. To simplify the notation, we let 
\[
	V = H^0(X, \OmX{1}) = H^0(A, \Omega_A^1) \quad \text{and} \quad
		S = \Sym(\Vd).
\]

\begin{proposition} \label{prop:complex}
With notation as above, we have
\[
	\Gr_{\bullet}^F  a_* (\OX, F) \simeq 
		\derR a_* \Bigl\lbrack 
			\OX \tensor S^{\bullet-g} \to \OmX{1} \tensor S^{\bullet-g+1} \to \dotsb \to
			\OmX{n} \tensor S^{\bullet-g+n} 
		\Bigr\rbrack,
\]
with differential induced by the evaluation morphism $\OX \tensor H^0(X, \OmX{1})
\to \OmX{1}$.  
\end{proposition}

\begin{proof}
The cotangent bundle of $A$ is isomorphic to the product $A \times V$, and so using
the notation from \subsecref{subsec:three}, we have $\shA_A = \OA \tensor S$ as well
as $a^* \shA_A = \OX \tensor S$. Consequently, $\omX \tensor a^* \shA_A$ can be
resolved by the complex
\[
	\Bigl\lbrack \AX^{\bullet-n} \to \OmX{1} \tensor \AX^{\bullet-n+1} \to 
		\dotsb \to \OmX{n} \tensor \AX^{\bullet} \Bigr\rbrack \tensor S
\]
placed as usual in degrees $-n, \dotsc, 0$. Applying Theorem~\ref{thm:laumon}, we find that
\[
	\Gr_{\bullet}^F a_* (\OX, F) \simeq
		\derR a_* \Bigl\lbrack \OX \tensor S^{\bullet-g} \to
		\OmX{1} \tensor S^{\bullet-g+1} \to \dotsb \to 
		\OmX{n} \tensor S^{\bullet-g+n} \Bigr\rbrack,
\]
with Koszul-type differential induced by the morphism $\OX \to \OmX{1} \tensor \Vd$, which
in turn is induced by the evaluation morphism $\OX \tensor V \to \OmX{1}$.
\end{proof}

On the other hand, we know from the discussion above that 
\[
	a_* \QHX \decal{n} \simeq \bigoplus_{i,j} M_{i,j} \decal{-i}.
\]
It follows that $\Gr_{\bullet}^F a_* (\OX, F)$ splits, as a complex of graded modules
over $\shA_A = \OA \tensor S$, into a direct sum of modules of the form
$\Gr_{\bullet}^F \Mmod_{i,j}$. Putting everything together, we obtain a key
isomorphism which relates generic vanishing, zero sets of holomorphic one-forms on
$X$, and the topology of the Albanese mapping.

\begin{corollary}\label{cor:key}
With notation as above, we have
\[
	\derR a_* \Bigl\lbrack \OX \tensor S^{\bullet-g} \to \OmX{1} \tensor
		S^{\bullet-g+1} \to \dotsb \to \OmX{n} \tensor S^{\bullet-g+n} \Bigr\rbrack
	\simeq \bigoplus_{i,j} \Gr_{\bullet}^F \Mmod_{i,j} \decal{-i}
\]
in the bounded derived category of graded $\OA \tensor S$-modules.
\end{corollary}	

Since the $M_{i,j}$ are Hodge modules on an abelian variety, it follows from
Theorem~\ref{hm_abelian} that each $\Gr_k^F \Mmod_{i,j} $ is a $\GV$-sheaf on $A$.
The isomorphism in Corollary~\ref{cor:key} shows that whether or not the
entire complex $\Gr_k^F \al(\OX, F)$ satisfies a generic vanishing theorem is
determined by the presence of nonzero $M_{i,j}$ with $i > 0$. We shall see in
\subsecref{subsec:nakano} how this leads to a generic vanishing theorem of Nakano-type.

\subsection{The defect of semismallness}

In this section, we describe the range in which the Hodge modules $M_i
= H^i \al \QHX \decal{n}$ can be nonzero. Recall the following definition, introduced
by de Cataldo and Migliorini \cite{dCM}*{Definition~4.7.2}.

\begin{definition}
The \define{defect of semismallness} of the map $a \colon X \to A$ is
\[
	\delta(a) = \max_{\ell \in \NN} \bigl( 2 \ell - \dim X + \dim A_{\ell} \bigr).
\]
where we set $A_{\ell} = \menge{y \in A}{\dim a^{-1}(y) \geq \ell}$ for $\ell \in
\NN$.
\end{definition}

To illustrate the meaning of this quantity, suppose that $X$ is of maximal
Albanese dimension, and hence that $\dim X = \dim a(X)$. In that case, $2 \ell - \dim
X + \dim A_{\ell} = 2 \ell - \codim \bigl( A_{\ell}, a(X) \bigr)$; thus, for
instance, $a$ is \emph{semismall} iff $\delta (a) = 0$. An important observation,
due to de Cataldo and Migliorini \cite{dCM}*{Remark~2.1.2}, is that $\delta(a)$
controls the behavior of $\al \QHX \decal{n}$ in the decomposition theorem.

\begin{proposition}[de Cataldo, Migliorini] \label{prop:vanishing}
With notation as above, 
\[
	\delta(a) = \max \Menge{k \in \NN}{H^k \al \QHX \decal{n} \neq 0}.
\]
In particular, one has $H^k \al \QHX \decal{n} = 0$ whenever $\abs{k} > \delta(a)$.
\end{proposition}

\begin{proof}
We begin by showing that $M_k \neq 0$ implies $k \leq \delta(a)$. Suppose that $M$ is
one of the simple summands of the Hodge module $M_k$,
which means that $M \decal{-k}$ is a direct factor of $\al \QHX \decal{n}$. The
strict support of $M$ is an irreducible subvariety $Z \subseteq A$, and there is a
variation of Hodge structure of weight $n + k - \dim Z$ on a dense open subset of $Z$
whose intermediate extension is $M$. Let $z \in Z$ be a general point, and let $i_z
\colon \{z\} \into A$ be the inclusion. Then $H^{-\dim Z} \iu_z M$ is the corresponding
Hodge structure, and so we get $H^{-\dim Z} \iu_z M \neq 0$. This implies that 
\[
	H^{-\dim Z + k} \iu_z \al \QHX \decal{n} \neq 0
\]
as well. Now let $F = a^{-1}(z)$ be the fiber. By the formula for proper base change
\cite{Saito-MHM}*{(4.4.3)}, we have $\iu_z \al \QHX \decal{n} \simeq \pl \QHF
\decal{n}$, where $p \colon F \to \{z\}$, and so 
\[
	H^{-\dim Z + k} \iu_z \al \QHX \decal{n} \simeq H^{n - \dim Z + k} \pl \QHF 
		= H^{n - \dim Z + k}(F, \QQ)
\]
is the usual mixed Hodge structure on the cohomology of the projective variety $F$.
By the above, this mixed Hodge structure is nonzero; it follows for dimension reasons
that $n - \dim Z + k \leq 2 \dim F$. Consequently, $k \leq \delta(a)$ as claimed.

To show that $M_{\delta(a)} \neq 0$, we proceed as follows. Let $\ell \geq 0$ be the
smallest integer for which $A_{\ell}$ has an irreducible component of dimension
$\delta(a) + n - 2\ell$. After choosing a suitable Whitney stratification of
$A$, we may assume that this irreducible component is the closure of a stratum $S
\subseteq A_{\ell}$ of dimension $\delta(a) + n - 2 \ell$. Let $i_S \colon S \into
A$ denote the inclusion, and define $Y = a^{-1}(S)$. Then $p \colon Y \to S$ is 
proper of relative dimension $\ell$, and so necessarily $H^{2 \ell}(Y) \neq 0$. As
before, we have
\[
	H^{2 \ell}(Y) = H^{2 \ell - n} \pl \QHY \decal{n} 
	\simeq H^{2 \ell - n} \iu_S \al \QHX \decal{n} \simeq
		\bigoplus_{i=-\delta(a)}^{\delta(a)} H^{2 \ell - n - i} \iu_S M_i.
\]
Now the conditions of support for the perverse sheaf $\rat M_i$ imply that
\[
	H^{2\ell-n-i} \iu_S M_i = 0 \quad \text{for $2 \ell - n - i > - \dim S$},
\]	
and so there must be at least one nonzero module $M_i$ with $2 \ell - n - i
\leq - \dim S$, or equivalently, $i \geq \delta(a)$. Since $M_i = 0$ for $i >
\delta(a)$, it follows that $M_{\delta(a)} \neq 0$.

The final assertion is a consequence of the isomorphism $M_{-k} \simeq M_k(k)$.
\end{proof}

\subsection{Generic vanishing on the Picard variety}
\label{subsec:nakano}

We now address the generic vanishing theorem of Nakano type, Theorem \ref{thm:nakano}
in the introduction, and related questions. Ideally, such a theorem would say that 
\[
V^q (\Omega_X^p) = \Menge{L \in \Pic^0(X)}{H^q \bigl( X, \OmX{p} \tensor L \bigr) \neq 0}
\]
has codimension at least $\abs{p+q-n}$ in $\Ah$. Unfortunately, such a good statement is
not true in general (see Example \ref{blow-up} below, following \cite{GL1}). To simplify our discussion
of what the correct statement is, we make the following definition.

\begin{definition}
Let $X$ be a smooth projective variety. We say that $X$ satisfies the \define{generic
Nakano vanishing theorem with index $k$} if 
\[
	\codim  V^q (\Omega_X^p)  \geq \abs{p+q-n} - k
\]
for every $p,q \in \NN$.
\end{definition}

Note that the absolute value is consistent with Serre duality, which implies that
\[
	V^q(\Omega_X^p) \simeq V^{n-q}(\Omega_X^{n-p}).
\]
We can use the analysis in \subsecref{subsec:albanese} to obtain a precise formula
for the index $k$, namely we show that $X$ satisfies generic Nakano vanishing with 
index $\delta(a)$, but not with index $\delta(a) -1$.

\begin{proof}[Proof of Theorem \ref{thm:nakano}]
Recall from Proposition~\ref{prop:vanishing} that we have
\[
	\delta(a) = \max \menge{k \in \ZZ}{H^k \al \QHX \decal{n} \neq 0}.
\]
It follows from the equivalence established in \cite{Popa}*{\S3} that generic Nakano
vanishing with index $k$ is equivalent to the statement that
\begin{equation} \label{eq:Nakano}
	\FM \Bigl( \derR a_* \OmX{p} \decal{n-p} \Bigr) \in \cDtcoh{\leq k}(\OAh).
\end{equation}
We shall prove that this formula holds with $k = \delta(a)$ by descending induction on $p
\geq 0$, starting from the trivial case $p = n+1$. To simplify the bookkeeping, we set
\[
	\shC_{\bullet} = \derR a_* \Bigl\lbrack \OX \tensor S^{\bullet-g} \to 
		\OmX{1} \tensor S^{\bullet-g+1} \to \dotsb \to \OmX{n} \tensor S^{\bullet-g+n} \Bigr\rbrack.
\]
In particular, recalling that $(\Mmod_i, F)$ is the filtered $\Dmod$-module underlying the
Hodge module $M_i = H^i a_* \QHX \decal{n}$, we get from Corollary \ref{cor:key} that
\[
	\shC_{g-p} = \derR a_* \Bigl\lbrack \OmX{p} \to \OmX{p+1} \tensor S^1 \to
		\dotsb \to \OmX{n} \tensor S^{n-p} \Bigr\rbrack
		\simeq \bigoplus_i \Gr_{g-p}^F \Mmod_i \decal{-i}.
\]
From Theorem~\ref{hm_abelian}, we know that $\FM(\Gr_{g-p}^F \Mmod_i) \in
\cCoh(\OAh)$, and so we conclude that $\FM \shC_{g-p} \in \cDtcoh{\leq k}(\OAh)$ for
$0 \leq p \leq n$.  It is clear from the definition of $\shC_{g-p}$ that there is a
distinguished triangle
\[
	\shC_{g-p}' \to \shC_{g-p} \to \derR a_* \OmX{p} \decal{n-p} 
		\to \shC_{g-p}' \decal{1},
\]
in which $\shC_{g-p}'$ is an iterated extension of $\derR a_* \OmX{r} \decal{n-r}$
for $p+1 \leq r \leq n$. From the inductive hypothesis, we obtain that 
$\FM \shC_{g-p}' \in \cDtcoh{\leq k}(\OAh)$. Now apply the functor $\FM$
to the distinguished triangle to conclude that \eqref{eq:Nakano} continues to hold
for the given value of $p$.

This argument can be reversed to show that $\delta(a)$ is the optimal value for the
index. Indeed, suppose that $X$ satisfies the generic Nakano vanishing theorem with
some index $k$. Since each complex $\shC_j$ is an iterated extension of $\derR a_*
\OmX{p} \decal{n-p}$, the above shows that $\FM \shC_j \in \cDtcoh{\leq k}(\OAh)$,
and hence that $\FM(\Gr_p^F \Mmod_i) = 0$ for every $p \in \ZZ$ and $i > k$. Because
the Fourier-Mukai transform is an equivalence of categories, we conclude that
$\Mmod_i = 0$ for $i > k$, which means that $\delta(a) \leq k$.
\end{proof}

\begin{example}\label{blow-up}
Our result explains the original counterexample from \cite{GL1}*{\S3}. 
The example consisted in blowing up an
abelian variety $A$ of dimension four along a smooth curve $C$ of genus at least two; if
$X$ denotes the resulting variety, then $a \colon X \to A$ is the Albanese mapping,
and a short computation shows that
\[
	H^2(X, \Omega_X^3 \tensor \au L) \simeq H^0(C, \omega_C \tensor L) \neq 0
\]
for every $L \in \Pic^0(A)$. This example makes it clear that the index in the
generic Nakano vanishing theorem is not equal to the dimension of the generic fiber
of the Albanese mapping. On the other hand, it is not hard to convince oneself that
\[
	H^k \al \QHX \decal{4} \simeq \begin{cases}
		\QHA \decal{4} &\text{for $k = 0$,} \\
		\QHC \decal{1} &\text{for $k = \pm 1$,} \\
		0 &\text{otherwise.}
	\end{cases}
\]	
Thus $\delta(a) = 1$ is the correct value for the index in this case.
\end{example}

Theorem \ref{thm:nakano}, combined with the equivalence between (1) and (3) in Theorem \ref{perverse_criterion},   
implies a vanishing result for the cohomology sheaves of the Fourier-Mukai transform of any $\Omega_X^p$, in analogy with the result for $\shO_X$ conjectured by Green and Lazarsfeld and proved in \cite{Hacon} (note again that 
$P^{-1} \simeq (1\times (-1))^* P$, so for vanishing $P$ and $P^{-1}$ can be used interchangeably).

\begin{corollary}
With the notation above, for every integer $p$ we have 
$$R^i \Phi_P \Omega_X^p = 0 \,\,\,\, {\rm for~all~} i < n - p - \delta(a) .$$ 
Moreover, $R^{n - p - \delta(a)} \Phi_P \Omega_X^p \neq 0$ for some $p$.
\end{corollary}

Let us finally note that Theorem \ref{thm:nakano} and its proof improve the previously known generic vanishing
results for $\Omega_X^p$ with $p < n$, and recover those for $\omega_X$ (or its
higher direct images):

First, it was proved in \cite{PP}*{Theorem 5.11} that
\[
\codim V^q (\Omega_X^p ) \ge |p + q  - n| - \mu (a),
\]
with $\mu (a) = {\rm max} \{k, m-1\}$, $k$ being the minimal dimension and $m$ the
maximal dimension of a fiber of the Albanese map.  It is a routine check that $\delta
(a) \le \mu (a)$.

Secondly, in the case of the lowest nonzero piece of the filtration on $a_*
(\OX, F)$ the situation is better than what comes out of 
Theorem~\ref{thm:nakano}. This allows one to recover the original generic vanishing theorem for $\omega_X$ of \cite{GL1}, as well as its extension to higher direct images $R^j a_* \omega_X$ given in \cite{Hacon}. Indeed, in the proof above note that 
\begin{equation}\label{lowest}
\shC_{g-n} = \derR a_* \omega_X \simeq  \bigoplus_i \Gr_{g-n}^F \Mmod_i
\decal{-i}.  
\end{equation}
This shows that $R^i a_* \omega_X \simeq \Gr_{g-n}^F \Mmod_i$. Since these sheaves
are torsion-free by virtue of Koll\'ar's theorem, it follows that $\Gr_{g-n}^F \Mmod_i = 0$
unless $0 \leq i \leq \dim X - \dim a(X)$. Thus one recovers the original generic
vanishing theorem of Green and Lazarsfeld. A similar argument
works replacing $\omega_X$ by higher direct images $R^i f_* \omega_Y$, where $f
\colon Y\rightarrow X$ is a projective morphism with $Y$ smooth and $X$ projective
and generically finite over $A$.

\begin{note}
It is worth noting that since for an arbitrary projective morphism $f \colon X \rightarrow
Y$ a decomposition analogous to (\ref{lowest}) continues to hold, one recovers the
main result of \cite{Kollar}*{Theorem~3.1}, namely the splitting of $\derR f_*
\omega_X$ in $\Dbcoh (Y)$. This is of course not a new observation: it is precisely
Saito's approach to Koll\'ar's theorem and its generalizations.
\end{note}

\section{Cohomological support loci for local systems}

\subsection{Supports of transforms}

Let $A$ be an abelian variety of dimension $g$, and set as before $V = H^0(A,
\Omega_A^1)$ and $S = \Sym \Vd$. Let $M$ be a mixed Hodge module on $A$, with
underlying filtered $\Dmod$-module $(\Mmod, F)$. Then $\Gr_{\bullet}^F \Mmod$ is a
finitely-generated graded module over $\Sym \shT_A \simeq  \OA \tensor S$, and we
denote the associated coherent sheaf on $T^{\ast} A = A \times V$ by $\shC(\Mmod,
F)$. We may then define the \define{total Fourier-Mukai transform} of
$\Gr_{\bullet}^F \Mmod$ to be
\[
	\FM \bigl( \Gr_{\bullet}^F \Mmod \bigr) 
		= \bigoplus_{k \in \ZZ} \FM(\Gr_k^F \Mmod),
\]
which belongs to the bounded derived category of graded modules over $\OAh \tensor
S$. (Note that the right hand side involves the standard Fourier-Mukai transform, and we are 
continuing to use the same notation for the total transform, as no confusion seems likely.)
The geometric interpretation is as follows: the object in $\Dbcoh \bigl(
\shO_{\Ah \times V} \bigr)$ corresponding to the total Fourier-Mukai transform of
$\Gr_{\bullet}^F \Mmod$ is
\[
	\FM \shC(\Mmod, F) 
		= \derR p_{23\ast} \Bigl( p_{13}^{\ast} \shC(\Mmod, F) \tensor p_{12}^{\ast} P \Bigr)
\]
where the notation is as in the following diagram:
\begin{diagram}{3em}{2.5em}
\matrix[math] (m) {
	A \times V & A \times \Ah \times V & \Ah \times V \\
	& A \times \Ah \\
}; 
\path[to] (m-1-2) edge node[above] {$p_{13}$} (m-1-1)
	(m-1-2) edge node[auto] {$p_{12}$} (m-2-2)
	(m-1-2) edge node[auto] {$p_{23}$} (m-1-3);
\end{diagram}

Let now $X$ be a smooth projective variety of dimension $n$,  let $a \colon X \to A$
its Albanese map, and consider again the decomposition 
\[
	a_* \QHX \decal{n} \simeq \bigoplus_{i,j} M_{i,j} \decal{-i} \in \Db \MHM(A).
\]
Denote by $\shC_{i,j} = \shC(\Mmod_{i,j}, F)$ the coherent sheaf on $T^{\ast} A = A
\times V$ determined by the Hodge module $M_{i,j}$. The supports in $\Ah \times V$ of
the total Fourier-Mukai transforms $\FM \shC_{i,j}$ are of a very special kind; this
follows by using a result of Arapura \cite{Arapura}. To state the result, we recall
the following term coined by Simpson \cite{Simpson2}*{p.~365}.

\begin{definition}
A \define{triple torus} is any subvariety of $\Ah \times H^0(A, \Omega_A^1)$ 
of the form 
\[
	\im \Bigl( \varphi^{\ast} \colon \hat{B} \times H^0(B, \Omega_B^1) 	
		\into \Ah \times H^0(A, \Omega_A^1) \Bigr),
\]
for a surjective morphism $\varphi \colon A \to B$ to another abelian variety $B$. A
subvariety is called a \define{torsion translate of a triple torus} if it is a
translate of a triple torus by a point of finite order in $A \times H^0(A, \Omega_A^1)$.
\end{definition}

\begin{proposition}\label{prop:linear}
With notation as above, every irreducible component of the support of $\FM \shC_{i,j}$
is a torsion translate of a triple torus in $\Ah \times V$.
\end{proposition}

\begin{proof}
It is convenient to prove a more general statement. For simplicity, denote by $\shC \in
\Dbcoh \bigl( \shO_{A \times V} \bigr)$ the object corresponding to $\Gr_{\bullet}^F
\al(\OX, F)$. Recall from Proposition~\ref{prop:complex} and Corollary~\ref{cor:key}
that we have
\begin{equation} \label{eq:recall-decomp}
	\shC \simeq \derR (a \times \id)_* \Bigl\lbrack 
		p_1^* \OX \to p_1^* \Omega_X^1 \to \dotsb \to p_1^* \Omega_X^n
	\Bigr\rbrack  \simeq \bigoplus_{i,j} \shC_{i,j} \decal{-i},
\end{equation}
where $p_1 \colon X \times V \to X$, and the complex in brackets is placed in degrees
$-n, \dotsc, 0$, and has differential induced by the evaluation map $\OX \tensor V
\to \Omega_X^1$.

Now define, for $m \in \NN$, the following closed subsets of $\Ah \times V$:
\begin{align*}
	Z^k(m) &= \Menge{(L, \omega) \in \Ah \times V}%
		{\dim \mathbf{H}^k \bigl(A,  L \tensor \shC \restr{ A \times \{\omega\} } \bigr) \geq m} \\
		&= \Menge{(L, \omega) \in \Ah \times V}%
		{\dim \mathbf{H}^k \bigl(X, L \tensor K(\Omega_X^1, \omega) \bigr) \geq m},
\end{align*}
where we denote the Koszul complex associated to a single one-form $\omega \in V$ by
\[
	K(\Omega_X^1, \omega) = \Bigl\lbrack \OX \xrightarrow{\wedge \omega}
		 \OmX{1} \xrightarrow{\wedge \omega} \dotsb \xrightarrow{\wedge \omega}
			\OmX{n} \Bigr\rbrack.
\]
It follows from \cite{Arapura}*{Corollary on p.~312} and \cite{Simpson2} that each
irreducible component of $Z^k(m)$ is a torsion translate of a triple torus.
To relate this information to the support of the total Fourier-Mukai transforms $\FM
\shC_{i,j}$, we also introduce 
\[
	Z_{i,j}^k(m) = \Menge{(L, \omega) \in \Ah \times V}%
		{\dim \mathbf{H}^k \bigl( L \tensor \shC_{i,j} \restr{ A \times \{\omega\} } \bigr) \geq m}.
\]
The decomposition in \eqref{eq:recall-decomp} implies that 
\[	
	Z^k(m) = \bigcup_{\mu} \, \bigcap_{i,j} Z_{i,j}^{k-i} \bigl( \mu(i,j) \bigr),
\]
where the union is taken over the set of functions $\mu \colon \ZZ^2 \to \NN$ with the
property that $\sum_{i,j} \mu(i,j) = m$. As in \cite{Arapura}*{p.~312}, this formula
implies that each irreducible component of $Z_{i,j}^k(m)$ is also a torsion translate
of a triple torus.

Finally, we observe that each irreducible component of $\Supp \bigl( \FM \shC_{i,j}
\bigr)$ must also be an irreducible component of one the sets $Z_{i,j}^k := Z_{i,j}^k
(1)$, which concludes the proof.  More precisely, we have
$$ \Supp \bigl( \FM \shC_{i,j} \bigr) = \bigcup_k Z^k_{i,j}.$$
Indeed, the base change theorem shows that $\Supp \bigl( R^k \Phi_P \shC_{i,j} \bigr) \subset Z^k_{i,j}$ for all $k$, with equality 
if $k \gg 0$. Now assume that $(L, \omega) \in Z^k_{i,j}$ is a general point of a component which is not contained in   
$\Supp \bigl( R^k \Phi_P \shC_{i,j} \bigr)$. We claim that then $(L, \omega) \in Z^{k+1}_{i,j}$, which concludes the proof by 
descending induction. If this were not the case, then again by the base change
theorem, the natural map $$R^k \Phi_P \shC_{i,j} \otimes \shO_{\Ah \times V}/\mm_{(L, \omega)} \to \mathbf{H}^k
\bigl(A,  L \tensor \shC_{i,j} \restr{ A \times \{\omega\} } \bigr)$$
would be surjective, which would contradict $(L, \omega) \notin
\Supp \bigl( R^k \Phi_P \shC_{i,j} \bigr)$.
\end{proof}

\subsection{Generic vanishing for rank one local systems}
We continue to use the notation from above. Let ${\rm Char} (X) = {\rm Hom} ( \pi_1 (X), \CC^{\ast})$ be the 
algebraic group of characters of $X$.
We are interested in bounding the codimension of the cohomological support loci 
\[
\Sigma^i (X) = \menge{\rho \in {\rm Char} (X)}{H^i (X, \CC_{\rho}) \neq 0},
\]
where $\CC_{\rho}$ denotes the local system of rank one associated to a character $\rho$.
The structure of these loci has been studied by Arapura \cite{Arapura} and Simpson \cite{Simpson2}, who showed that 
they are finite unions of torsion translates of algebraic subtori of ${\rm Char}
(X)$. This is established via an interpretation in
terms of Higgs line bundles. (We have already used the more precise result for Higgs
bundles from \cite{Arapura} during the proof of Proposition \ref{prop:linear}.)

We shall assume for simplicity that $H^2 (X, \ZZ)$ is torsion-free (though the
argument goes through in general, the only difference being slightly more complicated
notation). In this case, the space of Higgs line bundles on $X$ consisting of pairs
$(L, \theta)$ of holomorphic line bundles with zero first Chern class and holomorphic
one-forms can be identified as the complex algebraic group
\[
{\rm Higgs} (X) \simeq \widehat{A} \times V.
\]
One can define an isomorphism of real (but not complex) algebraic Lie groups
\[
{\rm Char} (X) \rightarrow {\rm Higgs} (X), \quad \rho \to (L_{\rho}, \theta_{\rho}),
\]
where $\theta_{\rho}$ is the $(1,0)$-part of ${\rm log} \|\rho \|$ interpreted as 
a cohomology class via the isomorphism $H^1 (X, \RR) \simeq {\rm Hom} ( \pi_1 (X),
\RR)$. 
\begin{note}
$L_{\rho}$ is not the holomorphic line bundle $\CC_{\rho} \tensor_{\CC}
\OX$, except when the character $\rho$ is unitary; this point is stated incorrectly in
\cite{Arapura}*{p.312}. Please refer to \cite{Simpson2}*{p.~364} or page~\pageref{pg:higgs} below for the 
precise description of $L_{\rho}$.
\end{note}

By means of this identification, the local system cohomology of $\CC_{\rho}$ can be
computed in terms of the Dolbeault cohomology of the Higgs bundle $(L_{\rho},
\theta_{\rho})$.

\begin{lemma}[\cite{Simpson}*{Lemma 2.2}]\label{lem:dolbeault}
There is an isomorphism
\begin{align*}
H^k (X, \CC_{\rho}) &\simeq 
		\mathbf{H}^{k-n} \bigl( L_{\rho} \tensor K(\Omega_X^1, \theta_{\rho}) \bigr)  \\
&=  \mathbf{H}^{k-n} \Bigl( X, \Bigl\lbrack L_{\rho} \xrightarrow{\wedge \theta_{\rho}}
L_{\rho}\otimes \OmX{1} \xrightarrow{\wedge \theta_{\rho}} \dotsb \xrightarrow{\wedge
\theta_{\rho}} L_{\rho}\otimes \OmX{n} \Bigr\rbrack
\Bigr),
\end{align*}
where the Koszul-type complex in brackets is again placed in degrees $-n, \dotsc, 0$.
\end{lemma}

We can now obtain the generic vanishing theorem for local systems of rank one, stated in
the introduction.

\begin{proof}[Proof of Theorem \ref{thm:locsyst}]
Note first that, since Verdier duality gives an isomorphism 
\[
	H^k(X, \CC_{\rho}) \simeq H^{2n-k}(X, \CC_{-\rho}),
\]
we only need to prove the asserted inequality for $k \geq n$. Furthermore,
Lemma~\ref{lem:dolbeault} shows that it is enough to prove, for $k \geq 0$, the
analogous inequality 
\[
	\codim_{{\rm Higgs}(X)} Z^k(1) \geq 2 \bigl( k - \delta(a) \bigr),
\]
for the subsets $Z^k (1) \subseteq {\rm Higgs}(X)$ that were introduced during
the proof of Proposition \ref{prop:linear}. Recall from there that
\[	
	Z^k(1) = \bigcup_{i,j} Z_{i,j}^{k -i}(1).
\]
Finally, we will see in a moment in the proof of Theorem \ref{thm:mCoh} that 
\[
	\codim \Supp R^{\ell} \Phi_P \shC_{i,j} \geq 2 \ell
\]
for all $\ell$. By the same base change argument as before, this is equivalent to
$\codim Z_{i,j}^{\ell}(1) \geq 2 \ell$ for all $\ell$. Since $\shC_{i,j}$ is zero
unless $i \leq \delta(a)$, we conclude that
\[
	\codim Z^k(1) \geq \min_i \codim Z_{i,j}^{k-i}(1) \geq 2 \bigl( k - \delta(a) \bigr),
\]
which is the desired inequality.
\end{proof}

\subsection{Duality and perversity for total transforms}

In this section, we take a closer look at the behavior of the object $\shC(\Mmod, F)$ 
under the total Fourier-Mukai transform. We begin with a result that shows how the
total Fourier-Mukai transform interacts with Verdier duality for mixed Hodge modules.

\begin{lemma} \label{lem:total-dual}
Let $M$ be a mixed Hodge module on $A$, let $M'$ be its Verdier dual, and let
$\shC(\Mmod, F)$ and $\shC(\Mmod', F)$ be the associated coherent sheaves on $A
\times V$. Then
\[
	\derR \shHom \bigl( \FM \shC(\Mmod, F), \shO_{\Ah \times V} \bigr) \simeq
		\iota^{\ast} \bigl( \FM \shC(\Mmod', F) \bigr),
\]
where $\iota = (-1_{\Ah}) \times (-1_V)$.
\end{lemma}

\begin{proof}
Recall that the Grothendieck dual on a smooth algebraic variety $X$ is given by
$\DX(\argbl) = \derR \shHom \bigl( \argbl, \omega_X \decal{\dim X} \bigr)$. To simplify the
notation, set $\shC = \shC(\Mmod, F)$ and $\shC' = \shC(\Mmod', F)$. Then
\begin{align*}
	\DAhV &\bigl( \FM \shC(\Mmod, F) \bigr)
		= \DAhV \Bigl( \derR p_{23\ast} \bigl( p_{13}^{\ast} \shC \tensor p_{12}^{\ast} P \bigr) \Bigr) \\
		&\simeq \derR p_{23\ast} \Bigl( p_{13}^{\ast} \bigl( \DAV \shC \bigr) \decal{g} 
			\tensor p_{12}^{\ast} P^{-1} \Bigr) \\
		&\simeq \derR p_{23\ast} \Bigl( 
			\bigl( \id \times \id \times (-1_V) \bigr)^{\ast}(p_{13}^{\ast} \shC') \tensor 
			\bigl( \id \times (-1_{\Ah}) \times \id \bigr)^{\ast} (p_{12}^{\ast} P) \Bigr) \decal{2g} \\
		&\simeq \bigl( (-1_{\Ah}) \times (-1_V) \bigr)^{\ast} \bigl( \FM \shC' \bigr) \decal{2g}
			= \iota^{\ast} \bigl( \FM \shC(\Mmod', F) \bigr) \decal{2g}.
\end{align*}
For the first isomorphism we  use Grothendieck duality, while for the second we use 
Theorem~\ref{thm:Verdier}.  Since $\dim \Ah \times V = 2g$, this implies the result.
\end{proof}

Now let $a \colon X \to A$ be the Albanese map of a smooth projective variety of
dimension $n$. Consider the coherent sheaves $\shC_{i,j} = \shC(\Mmod_{i,j}, F)$
that arise from the decomposition theorem applied to $\al \QHX \decal{n} \in \Db
\MHM(A)$; see Corollary~\ref{cor:key} for the notation. Many of the results of the
previous sections can be stated very succinctly in the following way, using the
second $t$-structure introduced in \subsecref{subsec:t-structure}.

\begin{theorem} \label{thm:mCoh}
With notation as above, each $\FM \shC_{i,j}$ is a $m$-perverse coherent
sheaf on $\Ah \times V$. More precisely, the support of the object $\FM \shC_{i,j}$
is a finite union of torsion translates of triple tori, subject to the inequality
\[
	\codim \Supp R^k \Phi_P \shC_{i,j} \geq 2k
\]
for every $k \in \ZZ$. Moreover, the dual objects
\[
	\derR \shHom \bigl( \FM \shC_{i,j}, \shO_{\Ah \times V} \bigr)
\]
have the same properties.
\end{theorem}

\begin{proof}
Let $q \colon \Ah \times V \to \Ah$ be the first projection. Since we are dealing
with sheaves of graded modules, the support of the quasi-coherent sheaf 
\[
	\ql \bigl( R^{\ell} \Phi_P \shC_{i,j} \bigr) = 
		\bigoplus_{k \in \ZZ} R^{\ell} \Phi_P \bigl( \Gr_k^F \Mmod_{i,j} \bigr)
\]
is the image of $\Supp R^{\ell} \Phi_P \shC_{i,j}$ under the map $q$. Thanks to
Theorem~\ref{hm_abelian}, each $\FM \bigl( \Gr_k^F \Mmod_{i,j} \bigr)$ is a
$c$-perverse coherent sheaf on $\Ah$, so each irreducible component of the image
has codimension at least $\ell$. On the other hand, by Proposition \ref{prop:linear},
the support of $R^{\ell} \Phi_P \shC_{i,j}$ is a finite union of translates of triple
tori. Since a triple torus is always of the form $\hat{B} \times H^0(B,
\Omega_B^1)$, we obtain $\codim \hat{B} \geq \ell$, and therefore
\[
	\codim \Supp R^{\ell} \Phi_P \shC_{i,j} \geq 2\ell.
\]
This implies that $\FM \shC_{i,j}$ belongs to $\mDtcoh{\leq 0} \bigl( \shO_{\Ah
\times V} \bigr)$. Since Verdier duality for mixed Hodge modules commutes with direct
images by proper maps, we have
\[
	\DA a_* \QHX \decal{n} \simeq a_* \DX \QHX \decal{n}
		\simeq a_* \QHX(n) \decal{n},
\]
which shows that each module $\DA(M_{i,j})$ is again one of the direct factors in
the decomposition of $a_* \QHX \decal{n}$. By Lemma~\ref{lem:total-dual}, the
dual complex is thus again of the same type, hence lies in $\mDtcoh{\leq 0}
\bigl( \shO_{\Ah \times V} \bigr)$ as well. It is then clear from the description of
the $t$-structure in \subsecref{subsec:t-structure} that both objects actually belong
to the heart $\mCoh(\shO_{\Ah \times V})$, hence are $m$-perverse coherent sheaves.
\end{proof}

\begin{note}\label{general_dmod}
More precisely, each $\FM(\shC_{i,j})$ belongs to the subcategory of $\mCoh(\shO_{\Ah
\times V})$ consisting of objects whose support is a finite union of translates of
triple tori. It is not hard to see that this subcategory---unlike the category of
$m$-perverse coherent sheaves itself---is closed under the duality functor $\derR
\shHom \bigl( \argbl, \shO_{\Ah \times V} \bigr)$, because each triple torus has even
dimension.  
\end{note}

\subsection{Perverse coherent sheaves on the space of holomorphic one-forms}

We conclude this part by observing that, in analogy with the method described in \cite{Popa},
our method also produces natural perverse coherent sheaves on the affine space
$V = H^0(A, \Omega_A^1)$, where $A$ is an abelian variety of dimension
$g$. We shall use the following projection maps:
\begin{diagram}{2em}{2.5em}
\matrix[math] (m) {
	A \times V & V & & \Ah \times V & V \\
	A & & & \Ah \\ 
};
\path[to] (m-1-1) edge node[auto] {$q$} (m-1-2)
	edge node[left] {$p$} (m-2-1);
\path[to] (m-1-4) edge node[auto] {$q$} (m-1-5)
	edge node[left] {$p$} (m-2-4);
\end{diagram}

\begin{lemma} \label{lem:FM-ample1}
Let $L$ be an ample line bundle on the abelian variety $\Ah$. For any object $E \in
\Dbcoh \bigl( \shO_{A \times V} \bigr)$, we have
\[
	\derR \shHom \bigl( \derR \ql ( \pu L \tensor \FM E ), \shO_V \bigr)
	 		\simeq \derR \ql \bigl( E' \tensor 
		\pu (-1_A)^{\ast} R^g \Psi_P(L^{-1}) \bigr),
\]
where we set
\[
	E' = (\id \times (-1_V))^{\ast} \derR \shHom 
		\bigl( E, \shO_{A \times V} \decal{g} \bigr).
\]
\end{lemma}

\begin{proof}
This is an exercise in interchanging Grothendieck duality with pushforward by proper
maps and pullback by smooth maps.
\end{proof}

We can essentially rephrase Lemma~\ref{lem:kodaira} as follows.

\begin{lemma} \label{lem:Coh-L}
Let $M$ be a mixed Hodge module on an abelian variety $A$, with underlying filtered
$\Dmod$-module $(\Mmod, F)$, and let $\shC(\Mmod, F)$ be the coherent sheaf on $A
\times V$ associated to $\Gr_{\bullet}^F \Mmod$. Then for every ample line bundle $L$
on $A$, one has 
\[
	\derR \ql \bigl( \pu L \tensor \shC(\Mmod, F) \bigr) \in \Coh(\OV).
\]
\end{lemma}

\begin{proof}
$V$ being affine, it suffices to prove that the hypercohomology of the complex is
concentrated in degree $0$. But this hypercohomology is equal to
\[
	\mathbf{H}^i \bigl(A \times V,  L \tensor \pl \shC(\Mmod, F) \bigr)
		\simeq \bigoplus_{k \in \ZZ} H^i \bigl( A, L \tensor \Gr_k^F \Mmod \bigr),
\]
which vanishes for $i > 0$ because of Lemma~\ref{lem:kodaira}.
\end{proof}

We can now obtain perverse coherent sheaves on the affine space $V$ by pushing
forward along the projection $q \colon A \times V \to V$.

\begin{proposition}
Let $M$ be a mixed Hodge module on $A$, with underlying filtered $\Dmod$-module
$(\Mmod, F)$, and let $\shC(\Mmod, F)$ be the coherent sheaf on $A \times V$
associated to $\Gr_{\bullet}^F \Mmod$. Then for every ample line bundle $L$ on
$\Ah$, one has
\[
	\derR \ql \Bigl( \pu L \tensor \FM \shC(\Mmod, F) \Bigr) \in \cCoh(\OV).
\]
\end{proposition}

\begin{proof}
By Lemma  \ref{lem:perverse}, it suffices to prove that
\[
	\derR \shHom \Bigl( \derR \ql \bigl( \pu L \tensor \FM \shC(\Mmod, F) \bigr), 
		\OV \Bigr) \in \Coh(\OV).
\]
By Lemma~\ref{lem:FM-ample1} and Theorem~\ref{thm:Verdier}, this object is isomorphic to
\begin{equation} \label{eq:object1}
	\derR \ql \Bigl( \shC(\Mmod', F) \tensor \pu R^g \Psi_P(L^{-1}) \Bigr),
\end{equation}
where $\shC(\Mmod', F)$ is associated to the Verdier dual $M' = \DA M$. Now we apply the
usual covering trick. Let $\varphi_L \colon \Ah \to A$ be the isogeny defined by $L$.
Then the object in \eqref{eq:object1} will belong to $\Coh(\OV)$ provided the same is true for
\begin{equation} \label{eq:object2}
	\derR \ql \Bigl( \varphi_L^{\ast} \shC(\Mmod', F)
			\tensor \pu L \Bigr) \tensor H^0 \bigl( \Ah, L \bigr).
\end{equation}
Since $\varphi_L^{\ast} \shC(\Mmod', F)$ comes from the mixed Hodge module
$\varphi_L^{\ast} M'$, this is a consequence of Lemma~\ref{lem:Coh-L}.
\end{proof}

\section{Strong linearity}

\subsection{Fourier-Mukai for $\Dmod$-modules}
\label{subsec:Rothstein}

A stronger version of the results on cohomological support loci in \cite{GL1} was
given in \cite{GL2} (see also various extensions in \cite{CH}). This is the statement
(SL) in the introduction;  it
states that the standard Fourier-Mukai transform $\derR \Phi_P \OX$ in $\Dbcoh
(\shO_{\widehat{A}})$ is represented by an (explicit) linear complex in the neighborhood of any point
in $\widehat{A}$. Here we extend this to the setting of the trivial $\Dmod$-module
$\OX$, and consequently to all the Hodge modules $M_{i,j}$ appearing in the previous
sections. In order to do this, we shall make use of the Fourier-Mukai transform for
$\Dmod$-modules on abelian varieties, introduced by Laumon \cite{Laumon2}
and Rothstein \cite{Rothstein}.

We start by setting up some notation. Let $A$ be a complex abelian variety of dimension $g$,
and let $\Ash$ be the moduli space of algebraic line bundles with integrable connection
on $A$. Note that $\Ash$ naturally has the structure of a quasi-projective algebraic
variety: on $\Ah$, there is a canonical vector bundle extension
\[
	0 \to \Ah \times H^0(A, \OmA{1}) \to E_{\Ah} \to \Ah \times \CC \to 0,
\]
and $\Ash$ is isomorphic to the preimage of $\Ah \times \{1\}$ inside $E_{\Ah}$. The
projection 
\[
\pi \colon \Ash \to \Ah,  \qquad (L, \nabla) \mapsto L
\] 
is thus a torsor for the trivial bundle $\Ah \times H^0(A, \OmA{1})$; this
corresponds to the fact that $\nabla + \omega$ is again an integrable connection for any
$\omega \in H^0(A, \OmA{1})$. Note that $\Ash$ is a group under tensor product, and
that the trivial line bundle $(\OA, d)$ plays the role of the zero element. 

Recall now that Laumon \cite{Laumon2} and Rothstein \cite{Rothstein} have extended
the Fourier-Mukai transform to $\Dmod$-modules. Their generalized Fourier-Mukai
transform takes bounded complexes of coherent algebraic $\Dmod$-modules on $A$ to
bounded complexes of algebraic coherent sheaves on $\Ash$; we briefly describe it
following the presentation in \cite{Laumon2}*{\S3}, which is more convenient for our
purpose.  On
the product $A\times \Ash$, the pullback $P^\natural$ of the Poincar\'e bundle $P$ is
endowed with a universal integrable connection $\nabla^\natural \colon P^\natural \to
\Omega_{A \times \Ash / \Ash}^1 \tensor P^\natural$, relative to $\Ash$.
Given any algebraic left $\mathcal{D}$-module $\Mmod$ on $A$, interpreted as a quasi-coherent
sheaf with integrable connection, we consider $p_1^* \Mmod \otimes P^\natural$ on $A\times
\Ash$, endowed with the natural tensor product integrable connection $\nabla$
relative to $\Ash$. We then define
\begin{equation}\label{eqn:laumon}
\derR \Phi_{P^\natural}  \Mmod := \derR p_{2\ast} {\rm DR} ( p_1^* \Mmod \otimes P^\natural, \nabla),
\end{equation}
where ${\rm DR} ( p_1^* \Mmod \otimes P^\natural, \nabla)$ is the usual (relative) de Rham complex 
\begin{equation*}
\left[ p_1^* \Mmod \otimes P^\natural \overset{\nabla}{\longrightarrow}  p_1^* \Mmod \otimes P^\natural \otimes \Omega^1_{A\times \Ash/ \Ash} \overset{\nabla}{\longrightarrow} \cdots \overset{\nabla}{\longrightarrow}
 p_1^* \Mmod \otimes P^\natural \otimes \Omega^g_{A\times \Ash/ \Ash} \right]
 \end{equation*}
placed in degrees $-g, \ldots,0$. As all of the entries in this complex are
relative to $\Ash$, it follows that $ \derR \Phi_{P^\natural} \Mmod$ is represented
by a complex of algebraic quasi-coherent sheaves on $\Ash$. Restricted to coherent
$\mathcal{D}$-modules, this is shown to induce an equivalence of categories 
$$\derR \Phi_{P^\natural} \colon \Dbcoh (\Dmod_A) \to  \Dbcoh (\OAsh).$$
The same result is obtained by a different method in \cite{Rothstein}*{Theorem 6.2}.

\begin{note}
Since $\Ash$ is not compact, it is essential to consider algebraic coherent sheaves
on $\Ash$ in the above equivalence. On $A$, this problem does not arise, because the
category of coherent analytic $\Dmod$-modules on a smooth projective variety is 
equivalent to the category of coherent algebraic $\Dmod$-modules by a version of the
GAGA theorem.
\end{note}

Now let $X$ be a smooth projective variety with Albanese map $a \colon X \to A$. 
By first pushing forward to $A$, or equivalently by working 
with the pullback of $(P^\natural, \nabla^\natural)$ to $X \times \Ash$, one can similarly define
$$\derR \Phi_{P^\natural} : \Dbcoh (\Dmod_X) \to  \Dbcoh (\OAsh).$$
In this and the following six subsections, our goal is to prove the following linearity
property for the Fourier-Mukai transform of the trivial $\Dmod$-module $\OX$ (see
Definition~\ref{def:linear} below for the definition of a linear complex
over a local ring).

\begin{theorem} \label{thm:linear}
Let $X$ be a smooth projective variety of dimension $n$, and let $\derR
\Phi_{P^\natural} \OX \in \Dbcoh(\OAsh)$ be the Fourier-Mukai transform of the trivial
$\Dmod$-module on $X$. Let $(L, \nabla) \in \Ash$ be an any point, and denote by $R =
\shO_{\Ash, (L, \nabla)}^{\mathit{an}}$ the local ring in the analytic topology. Then the stalk
$\derR \Phi_{P^\natural} \OX \tensor_{\OX} R$ is quasi-isomorphic to a linear complex
over $R$.
\end{theorem}

The proof of Theorem~\ref{thm:linear} takes up the following six subsections.
We will in fact prove the more precise version given in Theorem \ref{thm:linear_intro} in the introduction, 
by producing the natural analogue of the \emph{derivative complex} of \cite{GL2}.
As in that paper, the idea is to pull $\derR \Phi_{P^\natural} \OX$ back to a complex on the universal
covering space $H^1 (X, \CC) = H^1(A, \CC)$, and to use harmonic forms to construct a 
linear complex that is isomorphic to the pullback in a neighborhood of the preimage of the given
point. We will encounter, however, an important additional difficulty, namely that
harmonic forms are in general not preserved under wedge product.

We conclude this subsection by noting that, regardless of the explicit linear representation, 
in combination with Proposition~\ref{prop:summand} below we obtain that
every direct summand of $\derR \Phi_{P^\natural} \OX$ is also isomorphic to a linear
complex in an analytic neighborhood of any given point on $\Ash$.

\begin{corollary}\label{summands}
Suppose that a Hodge module $M$ occurs as a direct factor of some $H^i a_* \QHX
\decal{n}$. Let $(\Mmod, F)$ be the filtered holonomic $\Dmod$-module underlying $M$,
and let $\derR\Phi_{P^\natural} \Mmod \in \Dbcoh(\OAsh)$ be its Fourier-Mukai
transform. Then the stalk $\derR\Phi_{P^\natural} \Mmod \tensor_{\OX} R$ is
quasi-isomorphic to a linear complex over $R$.
\end{corollary}

\subsection{Analytic description}

We begin by giving an analytic description of the space $\Ash$, and of the
Fourier-Mukai transform for $\Dmod$-modules. For the remainder of the discussion, we
shall identify $H^1(A, \CC)$ with the space of translation-invariant complex
one-forms on the abelian variety $A$. As complex manifolds, we then have
\[
	\Ash \simeq H^1(A, \CC) \big/ H^1(A, \ZZ) \qquad \text{and} \qquad
		\Ah \simeq H^1(A, \OA) \big/ H^1(A, \ZZ),
\]
and this is compatible with the exact sequence
\[
	0 \to H^0(A, \OmA{1}) \to H^1(A, \CC) \to H^1(A, \OA) \to 0.
\]
In particular, the universal covering space of $\Ash$ is isomorphic to 
$H^1(A, \CC)$. Given a translation-invariant one-form
$\tau \in H^1(A, \CC)$, let $\tau^{1,0}$ be its holomorphic and
$\tau^{0,1}$ its anti-holomorphic part. Then the image of $\tau$ in $\Ash$
corresponds to the trivial bundle $A \times \CC$, endowed with the holomorphic
structure given by the operator $\dbar + \tau^{0,1}$ and the integrable
connection given by $\del + \tau^{1,0}$. We can summarize this by saying
that $\tau$ corresponds to the trivial smooth vector bundle $A \times \CC$, endowed
with the smooth integrable connection $\dtau = d + \tau$. 

The Albanese mapping $a \colon X \to A$ induces an isomorphism $H^1(A, \CC) = H^1(X,
\CC)$, under which translation-invariant one-forms on $A$ correspond to harmonic
one-forms on $X$ for any choice of K\"ahler metric. To shorten the notation, we set
$$W := H^1(A, \CC) = H^1(X, \CC).$$ 
Then a harmonic form $\tau \in W$ corresponds to the
trivial smooth line bundle $X \times \CC$, endowed with the smooth integrable
connection $\dtau = d + \tau$.

We can similarly interpret the pullback of the Poincar\'e bundle to the complex
manifold $X \times W$.  Let $\shC^k(X \times W / W)$ be the sheaf of smooth
complex-valued relative $k$-forms on $X \times W$ that are in the kernel of
$\dbar_W$, meaning holomorphic in the direction of $W$, and denote by $$C^k(X \times
W / W) := {p_2}_*  \shC^k(X \times W / W)$$ its pushforward to $W$.  Let $T \in
C^1(X \times W / W)$ denote the tautological relative one-form on $X \times W$: the
restriction of $T$ to the slice $X \times \{\tau\}$ is equal to $\tau$. Then the
pullback of $(P^\natural, \nabla^\natural)$ is isomorphic to the trivial smooth bundle $X
\times W \times \CC$, with complex structure given by the operator $\dbar_X + \dbar_W
+ T^{0,1}$, and with relative integrable connection given by the operator $\del_X +
T^{1,0}$. 

This leads to the following analytic description of the Fourier-Mukai transform
(similar to \cite{GL2}*{Proposition 2.3}). Let
\[
	D \colon C^k(X \times W / W) \to C^{k+1}(X \times W / W)
\]
be the differential operator defined by the rule $D(\alpha) = d_X \alpha + T \wedge
\alpha$. Using that $d_X T = 0$, it is easy to see that $D \circ D = 0$.

\begin{lemma}\label{smooth_description}
The complex of $\shO_W$-modules $\bigl( C^{\bullet}(X \times W / W), D \bigr)$
is quasi-isomorphic to the pullback $\pi^\ast \derR \Phi_{P^\natural} \OX$, where $\pi: W
\rightarrow \Ash$ is the universal cover.
\end{lemma}

\begin{proof}
By the definition of the Fourier transform, $\derR \Phi_{P^\natural} \OX$
is the derived pushforward, via the projection $p_2 \colon X \times \Ash \to \Ash$,
of the complex
\[
\DR(P^\natural, \nabla^\natural) = 
\left[ P^\natural \overset{\nabla^\natural}{\longrightarrow}  \Omega^1_{X \times \Ash/
\Ash} \otimes P^\natural \overset{\nabla^\natural}{\longrightarrow} \cdots
\overset{\nabla^\natural}{\longrightarrow} \Omega^g_{X\times \Ash/ \Ash} \tensor
P^\natural \right],
\]
where $(P^\natural, \nabla^\natural)$ denotes the pullback of the Poincar\'e bundle to $X
\times \Ash$. Since $\pi \colon W \to \Ash$ is a covering map, we thus obtain
\[
	\pi^\ast \derR \Phi_{P^\natural} \OX \simeq
		\derR p_{2\ast} \left( (\id \times \pi)^{\ast} \DR(P^\natural, \nabla^\natural) \right).
\]
As noted above, the pullback of $(P^\natural, \nabla^\natural)$ to $X \times W$ is
isomorphic to the trivial smooth bundle $X \times W \times \CC$, with complex
structure given by $\dbar_X + \dbar_W + T^{0,1}$, and relative integrable connection
given by $\del_X + T^{1,0}$. By a version of the Poincar\'e lemma, we therefore have
\[
	(\id \times \pi)^{\ast} \DR(P^\natural, \nabla^\natural) \simeq
		\bigl( \shC^{\bullet}(X \times W/W), D \bigr).
\]
To obtain the desired result, it suffices then to note that 
$$R^i {p_2}_* \shC^k(X \times W / W) = 0, ~{\rm for ~all~} k {\rm ~and~all~} i >0.$$ 
This follows from a standard partition of unity argument as in \cite{GH}*{p.~42}.
\end{proof}

To prove Theorem \ref{thm:linear_intro}, it now suffices to show that the stalk of the complex 
$\bigl( C^{\bullet}(X \times W / W), D \bigr)$ at any given point $\tau \in W$ is
quasi-isomorphic to the linear complex appearing in its statement. Choose a basis $e_1, \dotsc, e_{2g} \in H^1(X, \CC)$
for the space of harmonic one-forms on $X$, and let $z_1, \dotsc, z_{2g}$ be the
corresponding system of holomorphic coordinates on $W$, centered at the point $\tau$.
In these coordinates, we have
\[
	T = \tau + \sum_{j=1}^{2g} z_j e_j \in C^1(X \times W/W).
\]
Let $R = \shO_{W, \tau}^{\mathit{an}}$ be the analytic local ring at the point
$\tau$, with maximal ideal $\mm$ and residue field $R / \mm = \CC$. To simplify the
notation, we put 
\[
	C^k :=  C^k(X \times W / W) \tensor_{\shO_W} R.
\]
Then the $R$-module $C^k$ consists of all convergent power series of the form
\[
	\alpha = \sum_I \alpha_I \tensor z^I := \sum_{I \in \NN^{2g}} \alpha_I \tensor
		z_1^{I(1)} \dotsm z_{2g}^{I(2g)}
\]
where $\alpha_I \in A^k(X)$ are smooth complex-valued $k$-forms on $X$, and the
summation is over all multi-indices $I \in \NN^{2g}$. To describe the induced
differential in the complex, we define the auxiliary operators $\dtau = d + \tau$ and
\[
	e \colon C^k \to C^{k+1}, \qquad 
		e(\alpha) = \sum_{I, j} e_j \wedge \alpha_I \tensor z_j z^I.
\]
Then each differential $D \colon C^k \to C^{k+1}$ is given by the formula
\begin{equation} \label{eq:differential}
	D \alpha = \sum_I (d \alpha_I + \tau \wedge \alpha_I) \tensor z^I 
		+ \sum_{I,j} e_j \wedge \alpha_I \tensor z_j z^I
		= (\dtau + e) \alpha.
\end{equation}
Note that each $R$-module in the complex has infinite rank; moreover, in the formula
for the differential $D$, the first of the two terms is not linear in $z_1, \dotsc,
z_{2g}$. Our goal is then to build a linear complex quasi-isomorphic to
\[
	C^0 \to C^1 \to \dotsb \to C^{2n-1} \to C^{2n},
\]
by using harmonic forms with coefficients in the flat line bundle corresponding to
$\tau$. The space of $\dtau$-harmonic forms has the advantage of being
finite-dimensional; in addition, any such form $\alpha \in A^k(X)$ satisfies $\dtau
\alpha = 0$, and hence
\[
	D \alpha = \sum_{j=1}^{2g} e_j \wedge \alpha \tensor z_j.
\]
This shows that the differential is linear when restricted to the free $R$-module
generated by the $\dtau$-harmonic forms. The only problem is that we do not obtain a
subcomplex of $(C^{\bullet}, D)$ in this way, because the wedge product $e_j \wedge
\alpha$ is in general no longer harmonic. This difficulty can be overcome by
constructing a more careful embedding of the space of $\dtau$-harmonic $k$-forms into
$C^k$, as we now explain.

\subsection{Harmonic theory for flat line bundles}

In this section, we summarize the theory of harmonic forms with coefficients in a
flat line bundle, developed by Simpson \cite{Simpson}. Let $A^k(X)$ be the space of
smooth complex-valued $k$-forms on $X$.  Choose a K\"ahler metric on $X$, with
K\"ahler form $\omega \in A^2(X)$, and denote by
\[
	L \colon A^k(X) \to A^{k+2}(X), \qquad L(\alpha) = \omega \wedge \alpha
\]
the associated Lefschetz operator. The metric gives rise to the $\ast$-operator
\[
	\ast \colon A^k(X) \to A^{2n-k}(X),
\]
where $n = \dim X$, and the formula
\[
	(\alpha, \beta)_X = \int_X \alpha \wedge \ast \overline{\beta}
\]
defines a Hermitian inner product on the space $A^k(X)$. With respect to these inner
products, the adjoint of $L \colon A^k(X) \to A^{k+2}(X)$ is the operator $\Lambda
\colon A^k(X) \to A^{k-2}(X)$. Likewise, the adjoint of the exterior derivative $d
\colon A^k(X) \to A^{k+1}(X)$ is the operator $\dst \colon A^k(X) \to A^{k-1}(X)$,
described more explicitly as $\dst \alpha = - \ast d \ast \alpha$. We use the
notation $d = \del + \dbar$ for the decomposition of $d$ by type; thus $\del$ maps
$(p,q)$-forms to $(p+1,q)$-forms, and $\dbar$ maps $(p,q)$-forms to $(p,q+1)$-forms.

Now fix a holomorphic line bundle $L$ with integrable connection $\nabla$. Let $A^k(X, L)$
denote the space of smooth $k$-forms with coefficients in $L$. The theorem of
Corlette \cite{Simpson}*{Theorem~1} shows that there is (up to rescaling) a unique
metric on the underlying smooth bundle that makes $(L, \nabla)$ into a harmonic
bundle. Together with the K\"ahler metric on $X$, it defines Hermitian inner products
on the spaces $A^k(X, L)$. In the case at hand, the harmonic metric may be described
concretely as follows.  According to our previous discussion, there is a harmonic
one-form $\tau \in A^1(X)$,
such that $L$ is isomorphic to the trivial smooth bundle $X \times \CC$, with complex
structure given by $\dbar + \tau^{0,1}$ and integrable connection given by $\del +
\tau^{1,0}$. The harmonic metric on $L$ is then simply the standard metric on the
bundle $X \times \CC$.  Consequently, we have
\[
	A^k(X, L) = A^k(X),
\]
and both the $\ast$-operator and the inner product induced by the harmonic metric
agree with the standard ones defined above.

As before, we let $\dtau = d + \tau$ be the operator encoding the complex structure
and integrable connection on $(L, \nabla)$; concretely,
\[
	\dtau \colon A^k(X) \to A^{k+1}(X), \qquad 
		\alpha \mapsto d \alpha + \tau \wedge \alpha.
\]
Let $\dsttau \colon A^k(X) \to A^{k-1}(X)$ be the adjoint of $\dtau$ with respect
to the inner products, and let $\Laptau = \dtau \dsttau + \dsttau \dtau$ be the
Laplace operator; it is an elliptic operator of second order. If we denote by 
\[
	\Hartau^k = \ker \bigl( \Laptau \colon A^k(X) \to A^k(X) \bigr)
\]
the space of $\dtau$-harmonic $k$-forms, then $\Hartau^k$ is finite-dimensional, and
Hodge theory gives us a decomposition
\begin{equation} \label{eq:Hodge-dec}
	A^k(X) = \Hartau^k(X) \oplus \Laptau A^k(X),
\end{equation}
orthogonal with respect to the inner product on $A^k(X)$. Let $\Htau \colon A^k(X) \to
\Hartau^k$ be the orthogonal projection to the space of harmonic forms. It is not hard
to see that any $\alpha \in A^k(X)$ can be uniquely written in the form
\begin{equation} \label{eq:Greens}
	\alpha = \Htau \alpha + \Laptau \Gtau \alpha,
\end{equation}
where $\Gtau \alpha \in \Laptau A^k(X)$ is the so-called \define{Green's operator}. The
uniqueness of the decomposition implies that $\dtau \Gtau = \Gtau \dtau$ and $\dsttau
\Gtau = \Gtau \dsttau$.

Following Simpson, we have a decomposition $\dtau = \deltau + \dbartau$, where
\begin{align*}
	\deltau &= \del + \frac{\tau^{1,0} - \overline{\tau^{0,1}}}{2} 		
		+ \frac{\tau^{0,1} + \overline{\tau^{1,0}}}{2}, \\
	\dbartau &= \dbar + \frac{\tau^{0,1} - \overline{\tau^{1,0}}}{2} 		
		+ \frac{\tau^{1,0} + \overline{\tau^{0,1}}}{2}.
\end{align*}
The justification for defining these two peculiar operators is that they satisfy the
usual K\"ahler identities (which fail for the naive choice $\del + \tau^{1,0}$ and
$\dbar + \tau^{0,1}$).

\begin{theorem}[Simpson]
The following are true:
\begin{enumerate}
\item We have $\deltau \deltau = \dbartau \dbartau = \deltau \dbartau + \dbartau
\deltau = 0$.
\item Let $\delsttau$ and $\dbarsttau$ denote the adjoints of $\deltau$ and
$\dbartau$, respectively. Then the first-order K\"ahler identities
\[
	\delsttau = i \lbrack \Lambda, \dbartau \rbrack, \qquad
	\dbarsttau = -i \lbrack \Lambda, \deltau \rbrack, \qquad
	\dsttau = i \lbrack \Lambda, \dbartau - \deltau \rbrack.
\]
are satisfied.
\item We have $\dbartau \delsttau + \delsttau \dbartau = \deltau \dbarsttau +
\dbarsttau \deltau = 0$.
\item The Laplace operator satisfies
\[
	\Laptau = 2 (\deltau \delsttau + \delsttau \deltau)
		= 2 (\dbartau \dbarsttau + \dbarsttau \dbartau),
\]
and consequently, $\dtau$-harmonic forms are both $\deltau$-closed and $\delsttau$-closed.
\item We have $H \deltau = H \dbartau = H \delsttau = H \dbarsttau = 0$.
\item The Green's operator $\Gtau$ commutes with $\deltau$, $\dbartau$, $\delsttau$,
and $\dbarsttau$.
\end{enumerate}
\end{theorem}

\begin{proof}
It is easy to see from the definition that (1) holds. The first-order K\"ahler
identities in (2) are proved in \cite{Simpson}*{p.~14}; in this simple case, they can
also be verified by hand by a calculation on $\CC^n$ with the Euclidean metric. From
this, (3) and (4) follow as in the case of the usual K\"ahler identities
\cite{Simpson}*{p.~22}. Finally, (5) is a consequence of (4), and (6) follows from
the previous results by the uniqueness of the decomposition in \eqref{eq:Greens}.  
\end{proof}

\begin{note}
\label{pg:higgs}
The Higgs bundle associated to $(L, \nabla)$ is the smooth vector bundle $X \times
\CC$, with complex structure defined by the operator 
\[
	\dbar + \frac{\tau^{0,1} - \overline{\tau^{1,0}}}{2},
\]
and with Higgs field
\[
	\theta = \frac{\tau^{1,0} + \overline{\tau^{0,1}}}{2}.
\]
Note that $\theta$ is holomorphic on account of $\dbartau \dbartau = 0$. The
complex structure on the original flat line bundle is defined by the operator $\dbar
+ \tau^{0,1}$, which means that the two line bundles are different unless
$\tau^{0,1} = - \overline{\tau^{1,0}}$.
\end{note}

\begin{example}
Harmonic theory can be used to solve equations involving $\deltau$ (or any of the
other operators), as follows. Suppose that we are given an equation of the form
$\deltau \alpha = \beta$. A
necessary and sufficient condition for the existence of a solution $\alpha$ is that
$\deltau \beta = 0$ and $\Htau \beta = 0$. If this is the case, then among all possible
solutions, there is a unique one that is $\delsttau$-exact, namely $2 \delsttau \Gtau
\beta$. (In fact, this is the solution of minimal norm.) Note that we can always
define $\alpha = 2 \delsttau \Gtau \beta$; but since
\[
	\deltau \bigl( 2 \delsttau \Gtau \beta \bigr) = \beta - \Htau \beta - 2 \delsttau
		\Gtau (\deltau \beta),
\]
we only obtain a solution to the original equation when $\Htau \beta = 0$ and
$\deltau \beta = 0$. This idea will appear again in the construction below.
\end{example}

\subsection{Sobolev spaces and norm estimates}

At some point of the construction below, we will need to prove the convergence of
certain power series. This requires estimates for the norms of the two operators
$\dbartau$ and $\Gtau$ introduced above, which hold in suitable Sobolev spaces. Since
this is standard material in the theory of partial differential equations, we only
give the briefest possible summary; all the results that we use can be found, for example, in
\cite{Wells}*{Chapter~IV}.

From the K\"ahler metric on $X$, we get an $L^2$-norm on the space $A^k(X)$ of
smooth $k$-forms, by the formula
\[
	\norm{\alpha}_0^2 = (\alpha, \alpha)_X = \int_X \alpha \wedge \ast \bar{\alpha}.
\]
It is equivalent to the usual $L^2$-norm, defined using partitions of unity.
There is also a whole family of higher Sobolev norms: for $\alpha \in A^k(X)$, the
$m$-th order Sobolev norm $\norm{\alpha}_m$ controls the $L^2$-norms of
all derivatives of $\alpha$ of order at most $m$. The Sobolev space $W_m^k(X)$ is the
completion of $A^k(X)$ with respect to the norm $\norm{\argbl}_m$; it is a Hilbert
space. Elements of $W_m^k(X)$ may be viewed as $k$-forms $\alpha$ with measurable
coefficients, all of whose weak derivatives of order at most $m$ are
square-integrable.  Here is the first result from analysis that we need.

\begin{theorem}[Sobolev lemma]
If $\alpha \in W_m^k(X)$ for every $m \in \NN$, then $\alpha$ agrees almost
everywhere with a smooth $k$-form, and hence $\alpha \in A^k(X)$.
\end{theorem}

The second result consists of a pair of norm inequalities, one for the differential
operator $\delsttau$, the other for the Green's operator $\Gtau$.

\begin{theorem} \label{thm:estimates}
Let the notation be as above.
\begin{enumerate}
\item There is a constant $C > 0$, depending on $m \geq 1$, such that
\[
	\norm{\delsttau \alpha}_{m-1} \leq C \cdot \norm{\alpha}_m
\]
for every $\alpha \in W_m^k(X)$ with $m \geq 1$.
\item There is another constant $C > 0$, depending on $m \geq 0$, such that
\[
	\norm{\Gtau \alpha}_{m+2} \leq C \cdot \norm{\alpha}_m
\]
for every $\alpha \in W_m^k(X)$.
\end{enumerate}
\end{theorem}

\begin{proof}
The inequality in (1) is straightforward, using the fact that $\delsttau$ is a
first-order operator and $X$ is compact. On the other hand, (2) follows from the open
mapping theorem. To summarize the argument in a few lines, \eqref{eq:Hodge-dec}
is actually derived from an orthogonal decomposition
\[
	W_m^k(X) = \Hartau^k \oplus \Laptau W_{m+2}^k(X)
\]
of the Hilbert space $W_m^k(X)$. It implies that the bounded linear operator
\[
	\Laptau \colon W_{m+2}^k(X) \cap (\Hartau^k)^{\perp} \to 
		W_m^k(X) \cap (\Hartau^k)^{\perp}
\]
is bijective; by the open mapping theorem, the inverse must be bounded as well.
Since $\Gtau$ is equal to this inverse on $(\Hartau^k)^{\perp}$, and zero on
$\Hartau^k$, we obtain the desired inequality.
\end{proof}

\subsection{Construction of the linear complex}

We now return to the proof of Theorem~\ref{thm:linear_intro}. Recall that, after pullback
to the universal covering space $W$ of $\Ash$, the stalk of the Fourier-Mukai
transform of the $\Dmod$-module $\OX$ is represented by the complex of $R$-modules
\[
	C^0 \to C^1 \to \dotsb \to C^{2n-1} \to C^{2n},
\]
with differential
\[
	D(\alpha) = (\dtau + e)\alpha = \sum_I \dtau \alpha_I \tensor z^I
		+ \sum_{I,j} e_j \wedge \alpha_I \tensor z_j z^I.
\]
Our goal is to show that $\bigl( C^{\bullet}, D \bigr)$ is quasi-isomorphic to a
linear complex over $R$.

We begin by constructing a suitable linear complex from the finite-dimensional spaces
of $\dtau$-harmonic forms. Let $\bigl( \Hartau^{\bullet} \tensor R, \delta \bigr)$ be
the complex
\[
	\Hartau^0 \tensor R \to \Hartau^1 \tensor R \to \dotsb \to 
		\Hartau^{2n-1} \tensor R \to \Hartau^{2n} \tensor R,
\]
with differential obtained by $R$-linear extension from
\[
	\delta \colon \Hartau^k \to \Hartau^{k+1} \tensor R, \qquad
		\delta(\alpha) = \sum_j \Htau(e_j \wedge \alpha) \tensor z_j.
\]
One can see that this is indeed a complex by projecting to the harmonic subspace; as
a warm-up for later computations, we shall prove directly that $\delta \circ \delta =
0$.

\begin{lemma}
We have $\delta \circ \delta = 0$.
\end{lemma}

\begin{proof}
Before we begin the actual proof, let us make a useful observation: namely, that for
every $\alpha \in A^k(X)$, one has
\[
	\dtau(e_j \wedge \alpha) = d(e_j \wedge \alpha) + \tau \wedge e_j \wedge \alpha
		= - e_j \wedge d \alpha - e_j \wedge \tau \wedge \alpha = 
		- e_j \wedge \dtau \alpha,
\]
due to the fact that $e_j$ is a closed one-form. Now take $\alpha \in \Hartau^k$. Then
\[
	\Htau(e_j \wedge \alpha) = e_j \wedge \alpha - \Laptau \Gtau (e_j \wedge \alpha)
		= e_j \wedge \alpha - \dtau \dsttau \Gtau (e_j \wedge \alpha),
\]
because $\dtau(e_j \wedge \alpha) = 0$ by the above observation. Consequently,
\[
	e_k \wedge \Htau(e_j \wedge \alpha) =
		e_k \wedge e_j \wedge \alpha
		+ \dtau \bigl( e_k \wedge \dsttau \Gtau(e_j \wedge \alpha) \bigr),
\]
and since $\Htau \dtau = 0$, this allows us to conclude that
\[
	\delta(\delta \alpha) = \sum_{j,k} 
		\Htau \bigl( e_k \wedge \Htau(e_j \wedge \alpha) \bigr) \tensor z_j z_k
		= \sum_{j,k} \Htau(e_k \wedge e_j \wedge \alpha) \tensor z_j z_k = 0,
\]
which means that $\delta \circ \delta = 0$.
\end{proof}

Note that it is clear from the representation of cohomology via harmonic forms that the complex thus  constructed is quasi-isomorphic to the stalk at a point mapping to $(L, \nabla)$ of the complex appearing in the statement of 
Theorem \ref{thm:linear_intro}.

\subsection{Construction of the quasi-isomorphism} 

We shall now construct a sequence of maps $f^k \colon \Hartau^k \to C^k$,
in such a way that, after $R$-linear extension, we obtain a quasi-isomorphism $f
\colon \Hartau^{\bullet} \tensor R \to C^{\bullet}$. 
In order for the maps $f^k$ to define a morphism of complexes, the identity
\begin{equation} \label{eq:morphism}
	D f^k(\alpha) = f^{k+1}(\delta \alpha)
\end{equation}
should be satisfied for every $\alpha \in \Hartau^k$. As a first step, we shall find a 
\emph{formal} solution to the problem, ignoring questions of convergence for the time
being. Let $\hat{C}^k$ be the space of all formal power series
\[
	\alpha = \sum_I \alpha_I \tensor z^I :=
		 \sum_{I \in \NN^{2g}} \alpha_I \tensor z_1^{I(1)} \dotsm z_{2g}^{I(2g)}
\]
with $\alpha_I \in A^k(X)$ smooth complex-valued $k$-forms. We extend the various
operators from $A^k(X)$ to $\hat{C}^k$ by defining, for example,
\[
	\dtau \alpha = \sum_I \dtau \alpha_I \tensor z^I, \qquad
		e(\alpha) = \sum_{I,j} e_j \wedge \alpha_I \tensor z_j z^I, \qquad
		\text{etc.}
\]
Note that $C^k \subseteq \hat{C}^k$ is precisely the subspace of those power series
that converge in some neighborhood of $X \times \{z=0\}$.

To make sure that $f^k(\alpha)$ induces the correct map on cohomology, we
require that $\Htau f^k(\alpha) = \alpha$. Following the general
strategy for solving equations with the help of harmonic theory, we impose the
additional conditions $\deltau f^k(\alpha) = 0$ and  $\dbarsttau f^k(\alpha) = 0$.
Under these assumptions, \eqref{eq:morphism} reduces to
\[
	\dbartau f^k(\alpha) + e f^k(\alpha) = f^{k+1}(\delta \alpha).
\]
On account of $\dbarsttau f^k(\alpha) = 0$ and the K\"ahler identities, we should
then have
\begin{align*}
	f^k(\alpha) &= \Htau f^k(\alpha) + \Laptau \Gtau f^k(\alpha) 
		= \alpha + 2 \bigl( \dbartau \dbarsttau + \dbarsttau \dbartau \bigr) 
				\Gtau f^k(\alpha) \\
		&= \alpha + 2 \dbarsttau \Gtau \bigl( \dbartau f^k(\alpha) \bigr) 
		= \alpha - 2 \dbarsttau \Gtau \bigl( e f^k(\alpha) \bigr).
\end{align*}
This suggests that we try to solve the equation $\bigl( \id + 2 \dbarsttau \Gtau e
\bigr) f^k(\alpha) = \alpha$.

\begin{lemma} \label{lem:equation}
For any $\dtau$-harmonic form $\alpha \in \Hartau^k$, the equation 
\begin{equation} \label{eq:equation}
	\bigl( \id + 2 \dbarsttau \Gtau e \bigr) \beta = \alpha
\end{equation}
has a unique formal solution $\beta \in \hat{C}^k$.  This solution has the property
that $\Htau \beta = \alpha$, as well as $\dbarsttau \beta = 0$ and $\deltau \beta = 0$. 
\end{lemma}

\begin{proof}
To solve the equation formally, we write
\[
	\beta = \sum_{\ell=0}^{\infty} \sum_{\abs{I} = \ell} \beta_I \tensor z^I 
		= \sum_{\ell = 0}^{\infty} \beta_{\ell},
\]
making $\beta_{\ell}$ homogeneous of degree $\ell$ in $z_1, \dotsc, z_{2g}$. Taking
harmonic parts in \eqref{eq:equation}, it is clear that we must have $\beta_0 = \alpha$
and $\Htau \beta_{\ell} = 0$ for $\ell \geq 1$; for the rest, the equation dictates
that 
\begin{equation} \label{eq:solution}
	\beta_{\ell+1} = 
		- \sum_j 2 z_j \dbarsttau \Gtau \bigl( e_j \wedge \beta_{\ell} \bigr),
\end{equation}
for every $\ell \geq 0$, which means that there is a unique formal solution $\beta$.
It is apparent from the equation that $\dbarsttau \beta = 0$, and so to prove the lemma,
it remains to show that we have $\deltau \beta = 0$. Since $\deltau \beta_0 = 0$, we
can proceed by induction on $\ell \geq 0$. Using the K\"ahler identity $\deltau
\dbarsttau = - \dbarsttau \deltau$, and the fact that $\deltau(e_j \wedge
\beta_{\ell}) = - e_j \wedge \deltau \beta_{\ell}$, we deduce from
\eqref{eq:solution} that
\[
	\deltau \beta_{\ell+1} 
		= - \sum_j 2 z_j \deltau \dbarsttau \Gtau 
			\bigl( e_j \wedge \beta_{\ell} \bigr)
		= - \sum_j 2 z_j \dbarsttau \Gtau 
			\bigl( e_j \wedge \deltau \beta_{\ell} \bigr),
\]
and so $\deltau \beta_{\ell} = 0$ implies $\deltau \beta_{\ell+1} = 0$, as required.
\end{proof}

The next step is to prove the convergence of the power series defining the solution
to \eqref{eq:equation}. For $\eps > 0$, let
\[	
	W_{\eps} = \Menge{\tau + z \in W}{\sum_j \abs{z_j} < \eps},
\]
which is an open neighborhood of the point $\tau \in W$.

\begin{lemma}
There is an $\eps > 0$, such that for all $\alpha \in \Hartau^k$, the formal power
series
\[
	\beta = (\id + 2 \dbarsttau \Gtau e)^{-1} \alpha \in \hat{C}^k
\]
converges absolutely and uniformly on $X \times W_{\eps}$ to an element of $C^k(X
\times W_{\eps} / W_{\eps})$.
\end{lemma}

\begin{proof}
If we apply the estimates from Theorem~\ref{thm:estimates} to the relation in
\eqref{eq:solution}, we find that for every $m \geq 1$, there is a constant $C_m >
0$, such that
\begin{equation} \label{eq:estimate}
	\norm{\beta_{\ell+1}}_m 
		\leq C_m \sum_j \abs{z_j} \norm{\beta_{\ell}}_{m-1}
		\leq C_m \eps \cdot \norm{\beta_{\ell}}_{m-1}
\end{equation}
holds for all $\ell \geq 1$, provided that $z \in W_{\eps}$. Given that $\beta_0 =
\alpha$, we conclude that
\[
	\norm{\beta_{\ell}}_0 \leq (C_1 \eps)^{\ell} \norm{\alpha}_0.
\]
Now choose a positive real number $\eps < 1/C_1$. We then obtain
\[
	\sum_{\ell=0}^{\infty} \norm{\beta_{\ell}}_0 
		\leq \sum_{\ell=0}^{\infty} (C_1 \eps)^{\ell} \norm{\alpha}_0
		= \frac{\norm{\alpha}_0}{1 - C_1 \eps},
\]
from which it follows that $\beta$ is absolutely and uniformly convergent in the
$L^2$-norm as long as $z \in W_{\eps}$. To prove that $\beta$ is actually smooth, we
return to the original form of \eqref{eq:estimate}. It implies that, for any $m \geq
1$,
\[
	\sum_{\ell=0}^{\infty} \norm{\beta_{\ell}}_m
		\leq \norm{\alpha}_m + 
		C_m \eps \cdot \sum_{\ell=0}^{\infty} \norm{\beta_{\ell}}_{m-1}.
\]
By induction on $m \geq 0$, one now easily shows that $\sum_{\ell}
\norm{\beta_{\ell}}_m$ converges absolutely and uniformly on $X \times W_{\eps}$ for
every $m \geq 0$; because of the Sobolev lemma, this means that $\beta$ is smooth on
$X \times W_{\eps}$. Since we clearly have $\dbar_W \beta = 0$, it follows that
$\beta \in C^k(X \times W_{\eps}/W_{\eps})$.  
\end{proof}

The preceding lemmas justify defining
\[
	f^k \colon \Hartau^k \to C^k, \qquad
		f^k(\alpha) = (\id + 2 \dbarsttau \Gtau e)^{-1} \alpha.
\]
It remains to show that we have found a solution to the original problem
\eqref{eq:morphism}.

\begin{lemma}
For every $\alpha \in \Hartau^k$, we have $D f^k(\alpha) = f^{k+1}(\delta
\alpha)$.
\end{lemma}

\begin{proof}
Let $\beta = f^k(\alpha)$, so that $(\id + 2 \dbarsttau \Gtau e) \beta = \alpha$ and
$\deltau \beta = 0$. Noting that $\delta(\alpha) = \Htau(e \alpha)$, we need to show
that
\[
	(\id + 2 \dbarsttau \Gtau e) (\dbartau + e) \beta = \Htau(e \alpha).
\]
Since $e \circ e = 0$ and $e(\dbartau \beta) = - \dbartau(e \beta)$, we compute that
\[
	(\id + 2 \dbarsttau \Gtau e) (\dbartau + e) \beta
		= \dbartau \beta + e \beta + 2 \dbarsttau \Gtau e \dbartau \beta
		= \dbartau \beta + e \beta - 2 \dbarsttau \dbartau \Gtau e \beta.
\]
We always have $e \beta = \Htau(e \beta) + 2 \dbartau \dbarsttau \Gtau e \beta + 2
\dbarsttau \dbartau \Gtau e \beta$, and so we can simplify the above to
\begin{align*}
	(\id + 2 \dbarsttau \Gtau e) (\dbartau + e) \beta
		&= \dbartau \beta + \Htau(e \beta) + 2 \dbartau \dbarsttau \Gtau e \beta \\
		&= \Htau(e \beta) + \dbartau \bigl( \beta + 2 \dbarsttau \Gtau e \beta \bigr) 
		= \Htau(e \beta) + \dbartau \alpha = \Htau(e \beta).
\end{align*}
Thus it suffices to prove that $\Htau(e \beta) = \Htau(e \alpha)$. But this is
straightforward: from $\Htau(\beta) = \alpha$ and $\deltau \beta = 0$, we get
$\beta = \alpha + 2 \deltau \delsttau \Gtau \beta$, and therefore
\[
	e \beta = e \alpha - 2 \deltau \bigl( e \delsttau \Gtau \beta \bigr).
\]
Since $\Htau \deltau = 0$, we obtain the desired identity $\Htau(e \beta) = \Htau(e
\alpha)$.
\end{proof}

\begin{note}
The decomposition $\beta = \alpha + 2 \deltau \delsttau \Gtau \beta$ is the reason
for imposing the additional condition $\deltau f^k(\alpha) = 0$. Without this, it would be
difficult to relate the $\dtau$-harmonic parts of $e \alpha$ and $e \beta$ in the
final step of the proof.
\end{note}

If we extend $R$-linearly, we obtain maps of $R$-modules $f^k \colon \Hartau^k \tensor R
\to C^k$. Because \eqref{eq:morphism} is satisfied, they define a morphism of
complexes $f \colon \bigl( \Hartau^{\bullet} \tensor R, \delta \bigr) \to \bigl(
C^{\bullet}, D \bigr)$.

\begin{lemma}
$f \colon \Hartau^{\bullet} \tensor R \to C^{\bullet}$ is a quasi-isomorphism.
\end{lemma}

\begin{proof}
We use the spectral sequence \eqref{eq:SS-mm}. The complex $\Hartau^{\bullet} \tensor
R$ is clearly linear, and so the associated spectral sequence
\[
	{^1}E_1^{p,q} = \Hartau^{p+q} \tensor \Sym^p(\mm/\mm^2) \Longrightarrow
		H^{p+q} \bigl( \Hartau^{\bullet} \tensor R, \delta \bigr)
\]
degenerates at $E_2$ by Lemma \ref{lem:homog}. On the other hand, the complex
$C^{\bullet}$ also satisfies the conditions needed to  define (\ref{eq:SS-mm}), giving us a
second convergent spectral sequence with
\[	
	{^2}E_1^{p,q} = H^{p+q}(X, \ker \nabla) \tensor \Sym^p(\mm/\mm^2) \Longrightarrow
		H^{p+q} \bigl( C^{\bullet}, D \bigr),
\]
where $\ker \nabla$ is the local system corresponding to $(L, \nabla)$.
The morphism $f$ induces a morphism between the two spectral sequences; at $E_1$, it
restricts to isomorphisms ${^1}E_1^{p,q} \simeq {^2}E_1^{p,q}$, because $\Hartau^k
\simeq H^k(X, \ker \nabla)$. It follows that the second spectral sequence also
degenerates at $E_2$; it is then not hard to see that $f$ must
be indeed a quasi-isomorphism.
\end{proof}

We have now shown that the complex $\bigl( C^{\bullet}, D \bigr)$ is isomorphic, in
$\Dbcoh(R)$, to the linear complex $ \bigl( \Hartau^{\bullet} \tensor R, \delta \bigr)$. 
This completes the proof of Theorem~\ref{thm:linear_intro}.

\subsection{Filtered complexes and linear complexes}
\label{subsec:SS}
 
This section contains the homological algebra used in the constructions and proofs 
of the previous sections. It reviews and expands some of the content of \cite{LPS}*{\S1}, 
the main improvement with respect to that paper being Proposition \ref{prop:summand}.

\begin{definition} \label{def:linear}
Let $(R, \mm)$ be a regular local $k$-algebra of dimension $n$, with residue field $k = R / \mm$.
A \define{linear complex} over $R$ is a bounded complex $\bigl( K^{\bullet}, d \bigr)$ of
finitely generated free $R$-modules with the following property: there is a system of
parameters $t_1, \dotsc, t_n \in \mm$, such that every differential of the complex
is a matrix of linear forms in $t_1, \dotsc, t_n$.
\end{definition}

The property of being linear is obviously not invariant under isomorphisms (or
quasi-isomorphisms). To have a notion that works in the derived category $\Dbcoh(R)$,
we make the following definition.
\begin{definition}
We say that a complex is \define{quasi-linear} over $R$ if it is quasi-isomorphic to
a linear complex over $R$. 
\end{definition}

It is an interesting problem to try and find natural necessary and sufficient 
conditions for quasi-linearity in $\Dbcoh(R)$. For instance, if $E \in \Dbcoh(R)$,
then is it true that $E$ is quasi-linear over $R$ iff $E \tensor_R \Rh$ is quasi-linear over the
completion $\Rh$ with respect to the $\mm$-adic topology?

Here we only treat a special
case that was needed in Corollary~\ref{summands}, namely the fact that any direct summand 
in $\Dbcoh(R)$ of a quasi-linear complex  is again quasi-linear, plus a necessary condition 
for quasi-linearity.
The basic tool that we shall use is the notion of a minimal complex. Recall that a
bounded complex $(K^{\bullet}, d)$ of finitely generated free $R$-modules
is called \define{minimal} if $d(K^i) \subseteq \mm K^{i+1}$ for every
$i \in \ZZ$. This means that every differential of the complex is a matrix with
entries in the maximal ideal $\mm$. A basic fact is that every object in $\Dbcoh(R)$
is isomorphic to a minimal complex; moreover, two minimal complexes are isomorphic as
objects of $\Dbcoh(R)$ if and only if they are isomorphic as complexes (see 
 \cite{Roberts} for more details). Linear complexes are clearly minimal; it
follows that a minimal complex is quasi-linear iff it is isomorphic (as a
complex) to a linear complex.

\begin{proposition} \label{prop:summand}
Let $E \in \Dbcoh(R)$ be quasi-linear. If $E'$ is a direct summand of $E$ in the
derived category, then $E'$ is also quasi-linear.
\end{proposition}

\begin{proof}
Let $K^{\bullet}$ be a linear complex quasi-isomorphic to $E$, and
$L^{\bullet}$ a minimal complex quasi-isomorphic to $E'$. Since $E'$
is a direct summand of $E$ in the derived category, there are morphisms of complexes
\[
	s \colon L^{\bullet} \to K^{\bullet} \qquad \text{and} \qquad
		p \colon K^{\bullet} \to L^{\bullet}
\]
such that $p \circ s$ is homotopic to the identity morphism of $L^{\bullet}$. Because
$L^{\bullet}$ is minimal, it follows that $p \circ s$ reduces to the identity
modulo $\mm$, and is thus an isomorphism. After replacing $p$ by $(p \circ
s)^{-1} \circ p$, we may therefore assume without loss of generality that $p \circ s = \id$. 

Now choose a system of parameters $t_1, \dotsc, t_n \in \mm$ with the property that
each differential of the complex $K^{\bullet}$ is a matrix of linear forms in $t_1,
\dotsc, t_n$. We are going to construct a new complex $L_0^{\bullet}$
that has the same property, and is isomorphic to $L^{\bullet}$. Set $L_0^i =
L^i$, and define the differential $d_0 \colon L_0^i \to L_0^{i+1}$ to be the linear
part of $d \colon L^i \to L^{i+1}$. That is to say, let $d_0$ be the unique matrix of
linear forms in $t_1, \dotsc, t_n$ with the property that $(d - d_0)(L^i)
\subseteq \mm^2 L^{i+1}$ for every $i \in \ZZ$. It is easy to see that $d_0 \circ d_0
= 0$, and so $(L_0^{\bullet}, d_0)$ is a linear complex.

Likewise, we may define $s_0 \colon L_0^i \to K^i$ as the constant part of $s
\colon L^i \to K^i$; that is to say, as the unique matrix with entries in the field
$k$ such that $(s - s_0)(L^i) \subseteq \mm K^i$. Now consider the commutative diagram
\begin{diagram}{3em}{2.5em}
\matrix[math] (m) { 
	L^i & L^{i+1} \\
	K^i & K^{i+1} \\ 
}; 
\path[to] (m-1-1) edge node[auto] {$d$} (m-1-2)
		edge node[left] {$s$} (m-2-1);
\path[to] (m-1-2) edge node[auto] {$s$} (m-2-2);
\path[to] (m-2-1) edge node[auto] {$d$} (m-2-2);
\end{diagram}
By taking linear parts in the identity $d \circ s = s \circ d$, and using the fact
that $K^{\bullet}$ is a linear complex, we find that $d \circ s_0 = s_0 \circ d_0$;
consequently, $s_0 \colon L_0^{\bullet} \to K^{\bullet}$ is a morphism of complexes.
To conclude the proof, we consider the composition
\[
	p \circ s_0 \colon L_0^{\bullet} \to L^{\bullet}.
\]
By construction, $p \circ s_0$ reduces to the identity modulo $\mm$, and is
therefore an isomorphism. This shows that $L^{\bullet}$ is indeed
isomorphic to a linear complex, as claimed.
\end{proof}

A necessary condition for quasi-linearity is the degeneration of a certain spectral
sequence. To state this, we first recall some general facts.
Let $\bigl( K^{\bullet}, F \bigr)$ be a filtered complex in an abelian category. We
assume that the filtration is decreasing, meaning that $F^p K^n \supseteq F^{p+1}
K^n$, and satisfies
\[
	\bigcup_{p \in \ZZ} F^p K^n = K^n \qquad \text{and} \qquad
	\bigcap_{p \in \ZZ} F^p K^n = \{ 0 \}.
\]
Moreover, the differentials should respect the filtration, in the sense that
$d \bigl( F^p K^n \bigr) \subseteq F^p K^{n+1}$. Under these assumptions, the
filtered complex gives rise to a spectral sequence (of cohomological type)
\begin{equation} \label{eq:SS}
	E_1^{p,q} = H^{p+q} \bigl( F^p K^{\bullet} / F^{p+1} K^{\bullet} \bigr)
		\Longrightarrow H^{p+q} \bigl( K^{\bullet} \bigr).
\end{equation}
It converges by the standard convergence criterion \cite{McCleary}*{Theorem~3.2}. 

Going back now to the situation of a regular local $k$-algebra $(R, \mm)$ as above, on any bounded complex $K^{\bullet}$ of $R$-modules with finitely generated
cohomology, we can define the $\mm$-adic filtration by setting
\[
	F^p K^n = \mm^p K^n
\]
for all $p \geq 0$. Noting that $\mm^p / \mm^{p+1} \simeq \Sym^p(\mm/\mm^2)$, we have
\[
	F^p K^n / F^{p+1} K^n \simeq \bigl( K^n \tensor_R k \bigr) 
		\tensor_k \Sym^p(\mm/\mm^2).
\]
Provided that each $K^n$ has the property that
\[
	\bigcap_{p=1}^{\infty} \mm^p K^n = \{0\},
\]
the filtration satisfies the conditions necessary to define (\ref{eq:SS}), and we obtain a
convergent spectral sequence
\begin{equation} \label{eq:SS-mm}
	E_1^{p,q} = H^{p+q} \bigl( K^{\bullet} \tensor_R k \bigr) 
		\tensor_k \Sym^p(\mm/\mm^2)
		\Longrightarrow H^{p+q} \bigl( K^{\bullet} \bigr).
\end{equation}
It follows from the Artin-Rees theorem that the induced filtration on the limit is
$\mm$-good; in particular, the completion of $H^n \bigl( K^{\bullet} \bigr)$ with
respect to this filtration is isomorphic to $H^n \bigl( K^{\bullet} \bigr) \tensor_R
\Rh$. In this sense, the spectral sequence describes the formal analytic stalk of the
original complex.

\begin{lemma}[\cite{LPS}*{Lemma 1.5}]\label{lem:homog}
If $K^{\bullet}$ is a linear complex, then the spectral sequence 
\[
	E_1^{p,q} = H^{p+q} \bigl( K^{\bullet} \tensor_R k \bigr) 
		\tensor_k \Sym^p(\mm/\mm^2)
		\Longrightarrow H^{p+q} \bigl( K^{\bullet} \bigr)
\]
degenerates at the $E_2$-page.
\end{lemma}

Note that degeneration of the spectral sequence is not a sufficient condition for
being quasi-linear: over $R = k \llbracket x, y \rrbracket$, the complex
\begin{diagram}{2.5em}{2em}
\matrix[math] (m) { R & R^{\oplus 2} \\ }; 
\path[to] (m-1-1) edge node[auto] {$\begin{pmatrix} x^2 \\ y 
\end{pmatrix}$} (m-1-2);
\end{diagram}
is an example of a minimal complex that is not isomorphic to a linear complex, but
where the spectral sequence nevertheless degenerates at $E_2$.

\section{Perversity and the cohomology algebra}

\subsection{Perversity of the Laumon-Rothstein transform}
\label{subsec:perversity_laumon}

An important issue is to understand the exactness properties of the derivative-type 
complexes appearing in Theorem \ref{thm:linear_intro}.  In the case of the standard Fourier-Mukai 
transform $\derR \Phi_P \shO_X$, it was emphasized in \cite{LP} that this is a crucial 
step towards obtaining qualitative and quantitative information about the cohomology algebra of $X$. 

As our complexes locally represent the object $\derR \Phi_{P^\natural} \shO_X$, an
equivalent problem is to prove a vanishing statement for the higher direct images
$R^i \Phi_{P^\natural} \shO_X$.  This turns out to be related to the behavior of
$\derR \Phi_{P^\natural} \shO_X$ with respect to the $t$-structure $m$ defined in
\S\ref{subsec:t-structure}. Before rephrasing  Theorem \ref{thm:intro_structure} in
this language, note that due to the degree convention for the de Rham complex
associated to a $\Dmod$-module $\Mmod$, definition \ref{eqn:laumon} implies that
$R^i \Phi_{P^\natural} \shO_X = 0$ for $\abs{i} > n := \dim X$.

\begin{theorem}\label{thm:structure}
For any smooth projective complex variety $X$ we have
$$\derR \Phi_{P^\natural} \shO_X \in \mDtcoh{\leq \delta(a)}(\OX) \cap \mDtcoh{\geq
-\delta(a)}(\OX).$$
In particular, if the Albanese map of $X$ is semismall, $\derR \Phi_{P^\natural} \shO_X$ is 
an $m$-perverse coherent sheaf on $A^\natural$.
\end{theorem}
\begin{proof}
This is implied by our generic vanishing statement for local systems. Indeed, note that by Theorem \ref{thm:locsyst}, in the character space
${\rm Char}(X)$ we have
$$\codim \Sigma^{i +n} (X) \ge 2 \big(i - \delta(a)\big)$$
for any $i \ge 0$. 
Let us now define the locus 
$$S^i (X) : = \{(L, \nabla) ~|~ {\bf H}^i \big(X, {\rm DR}(L, \nabla)\big) \neq 0\} \subset X^\natural,$$
where $X^\natural$ denotes the moduli space of line bundles with integrable connection on $X$.
Via the well-known correspondence between bundles with integrable connection and local systems, we have a biholomorphic identification
$$F: A^\natural \simeq X^\natural \overset{\simeq}{\longrightarrow}  {\rm Char}(X), \,\,\,\, (L, \nabla) \mapsto {\rm Ker}(\nabla)$$
such that, as in the proof of Theorem \ref{thm:locsyst}, one has
$$F (S^i (X)) =   \Sigma^{i +n} (X).$$ 
This implies the inequalities 
$$\codim S^i (X) \ge 2 \big(i - \delta(a) \big),$$
and since by base change we have ${\rm Supp}~R^i \Phi_{P^\natural} \shO_X \subset S^i (X)$, we finally obtain
$$\derR \Phi_{P^\natural} \shO_X \in \mDtcoh{\leq \delta(a)}(\OX).$$
On the other hand, by the commutation of Verdier duality with the Fourier transform
(see \cite{Laumon2}*{Proposition 3.3.4}), the object $\derR \Phi_{P^\natural} \shO_X$
is self dual. Going back to the definition of the $t$-structure $m$, and using the
fact that each irreducible component of the support of $\derR \Phi_{P^\natural}
\shO_X$ has even dimension, this implies that $\derR \Phi_{P^\natural} \shO_X \in
\mDtcoh{\ge -\delta(a)}(\OX)$ as well.
\end{proof}

\begin{corollary}\label{cor:vanishing}
For any smooth projective complex variety $X$ we have
$$R^i \Phi_{P^\natural} \shO_X = 0 \,\,\,\, {\rm for~all~} i <  - \delta(a).$$
\end{corollary}
\begin{proof}
This follows immediately from Theorem \ref{thm:structure} and Lemma \ref{lem:heart}.
\end{proof}

\begin{note}
The second part of Theorem \ref{thm:structure} can be generalized by replacing $\shO_X$ (or rather $a_*\shO_X$) by any holonomic $\Dmod$-module on $A$; see Corollary \ref{cor:extension} below.
Due to this fact, the decomposition theorem, and Proposition \ref{prop:vanishing}, one can in fact 
draw an even stronger conclusion about $\derR \Phi_{P^\natural} \shO_X$, namely that
$$\derR \Phi_{P^\natural} \shO_X \simeq \bigoplus_{i = - \delta(a)}^{\delta(a)} E_i
\decal{-i},$$ with $E_i$ an $m$-perverse coherent sheaf on $A^\natural$ for each $i$. 
\end{note}

\subsection{Regularity of the cohomology algebra.}
\label{subsec:BGG}

The results of the previous section can be used to understand finer properties of the
singular cohomology algebra of $X$ as a graded module over the cohomology algebra of
$\Alb(X)$. We follow the method introduced in \cite{LP}. More precisely, we define
\[
	P_X = \bigoplus_{i=-n}^n P_{X, i} := \bigoplus_{i=-n}^n H^{n-i}(X, \CC).
\]
Via cup product, we may view this as a graded module over the exterior algebra 
\[
	E := \Lambda^{\ast} H^1 (X, \CC) \simeq H^{\ast} \bigl( \Alb(X), \CC \bigr).
\]
The conventions regarding the grading are similar to \cite{EFS}: in the algebra $E$,
elements of $H^1 (X, \CC)$ have degree $-1$; in $P_X$, the space $P_{X,i} = H^{n-i}
(X, \CC)$ is taken to have degree $i$. Via Poincar\'e duality, the dual graded
$E$-module is
\[
	\bigoplus_{i=-n}^n \Hom_{\CC} \bigl( P_{X,-i}, \CC \bigr) \simeq P_X.
\]
Using the results above, we can bound the degrees of the generators and syzygies of
$P_X$ as an exterior module.
To this end, recall the following analogue of Castelnuovo-Mumford regularity for
modules over the exterior algebra: given a graded $E$-module $Q$ generated in degrees
$\leq d$, one says that $E$ is \define{$m$-regular}
if the generators of $Q$ appear in degrees $d, d-1, \ldots, d-m$, the relations among
these generators are in degrees $d-1, \ldots, d-m-1$, and more generally the
$p$-th module of syzygies of $Q$ has all its generators in degrees $d-p, \dotsc,
d-m-p$. An equivalent condition (see \cite{LP}*{Proposition 2.2}) is the vanishing
\[
	\Tor_i^E(Q, \CC)_{-i-k} = 0
\]
for $i \in \ZZ$ and $k > m-d$.

\begin{proof}[Proof of Corollary~\ref{cor:intro_regularity}]
In view of Theorem \ref{thm:linear_intro} and Corollary \ref{cor:vanishing}, this
follows precisely as in \cite{LP}*{\S2}, and we will only briefly sketch the details.
On $W$ we have the complex $\mathcal{K}^\bullet$ of trivial vector bundles 
\[
	H^0(X, \shO_X) \tensor \OW \to H^1(X, \shO_X) \tensor \OW \to 
		\dotsb \to H^{2n}(X, \shO_X) \tensor \OW,
\]
placed in degrees $-n, \dotsc, n$, appearing in Theorem \ref{thm:linear_intro} for
$(L, \nabla) = (\OX, d)$. Passing to global sections, 
we obtain a complex $\mathbf{L}_X : = \Gamma (X, \mathcal{K}^\bullet)$, equal to
\[
	H^0(X, \shO_X) \tensor_{\CC} S \to H^1(X, \shO_X) \tensor_{\CC} S \to 
		\dotsb \to H^{2n}(X, \shO_X) \tensor_{\CC} S,
\]
where $S = {\rm Sym} (W^\vee)$. The Koszul-type differential of the complex
$\mathcal{K}^\bullet$ implies that $\mathbf{L}_X = \derL (P_X)$, the BGG-complex associated to the exterior module $P_X$ via the BGG correspondence; this is a linear complex of free $S$-modules. Since $W$ is an affine space, 
the exactness properties of $\mathbf{L}_X$ are dictated by those of $\mathcal{K}^\bullet$
around the origin (note that the differentials of $\mathcal{K}^\bullet$ scale
linearly along radial directions). But by Theorem \ref{thm:linear_intro}, the complex
$\mathcal{K}^\bullet$ represents $\derR \Phi_{P^\natural} \shO_X$ in an analytic
neighborhood of the origin. Corollary \ref{cor:vanishing} therefore implies that 
$\mathbf{L}_X$ is exact in cohomological degrees $< -\delta(a)$.
Since the dual graded $E$-module is again isomorphic to $P_X$,
\cite{Eisenbud}*{Theorem~7.8} shows that we have
\[
	\Tor_i^E(P_X, \CC)_{-i-k} = 0
\]
for $i \in \ZZ$ and $k > \delta(a)$. Since $P_X$ is generated in degrees $\leq n$,
it follows that $P_X$ is $\big(n+\delta(a)\big)$-regular as a graded $E$-module.
\end{proof}

The best possible situation is when the Albanese map of $X$ is semismall, in which case 
$P_X$ is $n$-regular. This, as well 
as the general result, is optimal. We check this in some simple examples.

\begin{example}\label{algebras}
Let $C$ be a smooth projective curve of genus $g \ge 2$, and denote by $C_n$ its
$n$-th symmetric product; we shall assume that $n \leq g-1$, so that the image of
$C_n$ in the Jacobian $J(C)$ is a proper subvariety. It is well known that the
Abel-Jacobi morphism $C_n \to J(C)$ is semismall, and small if $C$ is
non-hyperelliptic. The singular cohomology
of $C_n$ was computed by Macdonald \cite{MacDonald}*{(6.3)}, with the following
result: There is a basis $\xi_1, \dotsc, \xi_{2g} \in H^1(C_n, \ZZ)$, and an element
$\eta \in H^2(C_n, \ZZ)$, such that $H^{\ast}(C_n, \ZZ)$ is generated by $\xi_1,
\dotsc, \xi_{2g}$ and $\eta$, subject only to the relations
\[
	\xi_{i_1} \dotsm \xi_{i_a} \xi_{j_1+g} \dotsm \xi_{j_b+g}
	(\xi_{k_1} \xi_{k_1+g} - \eta) \dotsm (\xi_{k_c} \xi_{k_c+g} - \eta) \eta^d= 0,
\]
where $i_1, \dotsc, i_a, j_1, \dotsc, j_b, k_1, \dotsc, k_c$ are distinct integers
in $\{1, \dotsc, g\}$, and $ a + b + 2c + d = n+1$. 
In the notation introduced above, this implies that $P_{C_n}$ is
generated as a graded $E$-module by the elements $ \eta^k$, with $k = 0, \ldots, n$,
while the relations imply that for $k > n/2$, these generators are obtained from 
those in lower degrees. It follows that $P_{C_n}$ is actually generated in degrees $0, \dotsc, n$.
On the other hand, note that when $n$ is even, $\eta^{n/2}$ is indeed a new 
generator in degree $0$, not coming from lower degrees.
This shows that in this case Corollary~\ref{cor:intro_regularity} is sharp. Equivalently,  up to (and including) the degree $-1$ term, the complex $\derL_{C_n}$ is exact. Is obtained as a direct sum of Koszul complexes in the vector space $W$, with a new one starting in each even degree; in other words, at a point $w\in W$, its first half looks as follows:
$$0 \to \CC \to W \to \wedge^2 W \oplus \CC \to \wedge^3 W \oplus W  \to 
\wedge^4 W \oplus \wedge^2 W \oplus \CC \to \cdots$$
where the differentials are given by wedging with $w$.
\end{example}

\begin{example}
Let us revisit Example \ref{blow-up} from this perspective. We are considering a four-fold $X$ obtained by blowing up 
an abelian variety $A$ of dimension $4$ along a smooth projective curve $C\subset A$ of genus at least $2$. We have
seen that $\delta (a) = 1$. It is not hard to compute the singular cohomology of $X$. For instance, 
denoting by $E$ the exceptional divisor on $X$, and $W = H^1 (A, \CC)$, we have 
\begin{align*}
H^1 (X, \CC) &\simeq W \\
H^2 (X, \CC) &\simeq H^2 (A, \CC) \oplus H^0(E, \CC) \simeq \wedge^2 W \oplus \CC \\
H^3 (X, \CC) &\simeq H^3 (A, \CC) \oplus H^1 (E, \CC) \simeq \wedge^3 W \oplus H^1
(C, \CC) \\
H^4 (X, \CC) &\simeq H^4 (A, \CC) \oplus H^2 (E, \CC) \simeq \wedge^4 W \oplus H^2
(C, \CC) \oplus \CC.
\end{align*}
Since the differentials are given by wedging with $1$-forms, 
the complex $\derL_X$  is then, up to the $H^4$-term, the direct sum of the (truncation of) the Koszul complex corresponding to $A$, and 
the following sequence corresponding to the cohomology of $E$, starting at $H^2$:
$$0 \rightarrow \CC \rightarrow \CC^{2g} \rightarrow \CC\oplus \CC.$$
This is exact at the first step, but clearly not at the second, which shows that
$\derL_X$ is exact at the first three terms, as in Corollary \ref{cor:intro_regularity},
but not at the fourth. Equivalently, $P_X$ is $5$-regular, but not $4$-regular.
\end{example}

Finally, we note that the partial exactness of the complex $\derL_X$ in Corollary \ref{cor:intro_regularity} would have numerous quantitative
applications related to the Betti numbers of $X$, again as in \cite{LP}, in the case
when $R^{-\delta(a)} \Phi_{P^\natural} \shO_X$ is locally 
free in a punctured neighborhood of the origin. At the moment we do not have a good understanding of geometric conditions that would 
imply this. It is however not hard to see that the absence of irregular fibrations for $X$ does not suffice (as in the case of the standard 
Fourier-Mukai transform), even when the Albanese map is semismall.

\section{Generalizations and open problems}

\subsection{Extensions to $\Dmod$-modules on abelian varieties}

In light of our results, it should be clear that all the theorems in generic
vanishing theory are in reality statements about a certain class of filtered
$\Dmod$-modules on abelian varieties, namely those underlying mixed Hodge modules.
Moreover, the dimension (D), linearity (L), and strong linearity (SL) results that we
have discussed should be viewed as properties of the Fourier-Mukai transforms of such
$\Dmod$-modules.

A natural question is then whether there is a larger (and more easily described)
class of $\Dmod$-modules on abelian varieties for which the same results are true.
It is known that when $(\Mmod, F)$ underlies a mixed Hodge module $M$,
the $\Dmod$-module $\Mmod$ is always regular and holonomic; when $M$ is pure, $\Mmod$
is in addition semi-simple. This suggests that the results of generic vanishing theory might extend
to $\Dmod$-modules with those properties on abelian varieties.

Using the recent work of Sabbah \cites{Sabbah1,Sabbah2} and Mochizuki
\cites{Mochizuki1,Mochizuki2} on the correspondence between semi-simple holonomic
$\Dmod$-modules and polarizable twistor $\Dmod$-modules, such an extension is indeed
possible. The results are as follows.

\begin{theorem}
Let $\Mmod \in \Db_h(\Dmod_A)$ be a complex of $\Dmod$-modules with bounded holonomic
cohomology on a complex abelian variety $A$. Then for every $k,m \in \ZZ$, the
cohomological support locus
\[
	S_m^k(\Mmod) := \Menge{(L, \nabla) \in \Ash}%
		{\dim \mathbf{H}^k \Bigl( A, \DR_A \bigl( \Mmod 
			\tensor_{\OA} (L, \nabla) \bigr) \Bigr) \geq m}
\]
is a finite union of translates of triple tori in $\Ash$; the translates are by
torsion points when $\Mmod$ is of geometric origin. The cohomological support loci
satisfy
\[
	\codim S_m^k(\Mmod) \geq 2k \qquad \text{for every $k \in \ZZ$}
\]
in the special case when $\Mmod$ is a single holonomic $\Dmod$-module.
\end{theorem}

The theorem implies the analogous result for cohomological support loci of
constructible complexes and perverse sheaves, which are now subsets of ${\rm
Char}(A)$; this is because of the Riemann-Hilbert correspondence. By the usual base
change arguments, one derives the following properties of the Fourier-Mukai transform.

\begin{corollary}\label{cor:extension}
For $\Mmod \in \Db_h(\Dmod_A)$, the support of the Fourier-Mukai transform
$\derR \Phi_{P^\natural} \Mmod \in \Dbcoh(\OAsh)$ is a finite union of translates of
triple tori in $\Ash$; the translates are by torsion points when $\Mmod$ is of
geometric origin. When $\Mmod$ is a single holonomic $\Dmod$-module, the system of
inequalities
\[	
	\codim \Supp R^{\ell} \Phi_{P^\natural} \Mmod \geq 2 \ell \qquad
		\text{for every $\ell \in \ZZ$}
\]
is satisfied, which implies that $\derR \Phi_{P^\natural} \Mmod$ is an $m$-perverse
coherent sheaf on $\Ash$.
\end{corollary}

For semi-simple holonomic $\Dmod$-modules, there is also a result analogous to (SL).

\begin{theorem}
If $\Mmod$ is a semi-simple holonomic $\Dmod$-module on $A$, then the Fourier-Mukai
transform $\derR \Phi_{P^\natural} \Mmod$ is locally, in the analytic topology,
quasi-isomorphic to a linear complex constructed from the cohomology of twists of $\Mmod$.
\end{theorem}

The proofs will appear elsewhere. They are based on an extension of the
Fourier-Mukai transform to twistor $\Dmod$-modules, which produces complexes of
coherent analytic sheaves on the twistor space $\operatorname{Tw}(\Ash)$ of the
quaternionic manifold $\Ash$, and on the powerful results about twistor
$\Dmod$-modules by Mochizuki and Sabbah. This treatment unifies the results for the
two spaces $\Ah \times H^0(A, \Omega_A^1)$ and $\Ash$ that we used in this paper:
they appear as two different fibers of $\operatorname{Tw}(\Ash) \to
\PPn{1}$. (This can also be used to deduce that the main result of \cite{GL2} is 
implied by Theorem \ref{thm:linear_intro} by passage to graded pieces.)

\subsection{Open problems}

Given the discussion above, a very interesting problem in this context is the following:

\begin{problem}
Describe classes of $\Dmod$-modules---such as holonomic, regular holonomic,
or semi-simple holonomic $\Dmod$-modules---on an abelian variety in terms of their
Fourier-Mukai transforms. In other words, characterize the subcategories of
$\Dbcoh(\OAsh)$ that correspond to the categories of such $\Dmod$-modules under
$\derR \Phi_{P^\natural}$.
\end{problem}

The two results above give several necessary conditions, so the problem is really to
find sufficient conditions. Such a description might also shed some light on the
difficult question of which $\Dmod$-modules are of geometric origin.

\bigskip

One may also wonder whether there are extensions of various results in this paper in the 
non-projective setting.

\begin{problem}
Does the analogue of Theorem \ref{hm_abelian} hold on arbitrary complex tori?
\end{problem}

Note that the (SL) type result, Theorem \ref{thm:linear_intro}, generalizes to compact K\"ahler 
manifolds, since the proof only uses harmonic theory. This raises the question whether  our
statements of type (D), here relying heavily on vanishing 
theorems for ample line bundles, extend to that context as well.

\begin{problem}
Are there analogues of the generic vanishing theorems \ref{thm:nakano}, \ref{thm:locsyst} 
and \ref{thm:intro_structure} in the K\"ahler setting?
\end{problem}

Finally, with respect to the discussion at the end of \S\ref{subsec:BGG}, it is natural to address the following:

\begin{problem}
Find geometric conditions on $X$ under which the higher direct image $R^{-\delta(a)} \Phi_{P^\natural} \shO_X$ is 
locally free in a punctured neighborhood of the origin. As a stronger question, find such conditions under which 
cohomological support loci $\Sigma^i (X)$ contain the origin as an isolated point for
all $i > n - \delta(a)$.
\end{problem} 

\section*{References}



\end{document}